\documentclass[12pt]{amsart}
\usepackage{tikz, forest, amscd, amssymb, soul}
\usetikzlibrary{trees}
\usetikzlibrary{decorations.markings} 
\usepackage{graphicx} 
\graphicspath{{figures/}}  
\usepackage{soul} 
\def\Z {\mathbb{Z}} 
 
\def\R {\mathbb{R}}  
\def\C {\mathbb{C}}

\def \co{\colon\!}

\def \ftnote{\let\thefootnote\relax\footnotetext}

\def\Imm{{\mathop\mathrm{Imm}}}

\newcommand{\hra}{\hookrightarrow}
\newcommand{\imra}{\looparrowright}

\usepackage{tikz-cd}
\usepackage{color}
\usepackage{subcaption}

\newtheorem{theorem}{Theorem}

\newtheorem{lemma}[theorem]{Lemma}
\newtheorem{corollary}[theorem]{Corollary}

\theoremstyle{remark}
\newtheorem{remark}{Remark} 
\newtheorem{example}{Example}

\title[]{On non-orientable surfaces embedded in 4-manifolds}
\date{\today}

\begin{document}
\author{David Auckly, Rustam Sadykov}
\thanks{The first author is partially supported by Simons Foundation grant 585139, and National Science Foundation grant DMS 1952755.}
\begin{abstract}

We find conditions under which a non-orientable closed surface  smoothly embedded 
into an orientable  $4$-manifold $X$  can be represented
by a connected sum of an embedded closed surface in $X$ and 
an unknotted projective plane in a $4$-sphere. This allows us to extend 
the Gabai $4$-dimensional light bulb theorem and the 
Auckly-Kim-Melvin-Ruberman-Schwartz ``one is enough" theorem to the case 
of non-orientable surfaces.   
\end{abstract}
\maketitle

\section{Introduction}

The goal of the present note is to determine conditions under which a non-orientable closed surface $S$ smoothly embedded into a $4$-manifold $X$ admits a splitting into the connected sum of an embedded  surface $S'$ and an unknotted projective plane $P^2$ in a $4$-sphere,  
 i.e., there exists a diffeomorphism of pairs
\begin{eqnarray}\label{eq:first}
      (X, S) \cong  (X, S')\# (S^4, P^2).
\end{eqnarray}

The important ingredient for the existence of the splitting is the existence of a geometric dual for $S$. 
We say that a sphere $G$ embedded into $X$ is a \emph{geometric dual} for a connected surface $S$ if the normal Euler number of $G$ is trivial and $G$ intersects $S$ transversally at a unique point. We note that some authors, including Gabai~\cite{Ga} and Freedman-Quinn~\cite{FQ}, use the term (embedded) \emph{transverse sphere} to refer to a geometric dual.
The surface $S$ is \emph{$G$-inessential} if the induced homomorphism $\pi_1(S\setminus G)\to \pi_1(X\setminus G)$ is trivial. When there is such a geometric dual, any $P^2$ summand in $S$ can be split off. 

\begin{theorem}\label{th:1a}  Let $S$ be a connected non-orientable closed surface embedded in an orientable $4$-manifold $X$. Suppose that $S$ is $G$-inessential for a geometric dual $G$. Let $P^2$ be a projective plane summand of $S$. 
Then the pair $(X, S)$ splits as in (\ref{eq:first}) with $P^2$ unknotted, and with the surface $S'$ still $G$-inessential for the geometric dual $G$. 
\end{theorem}

\begin{remark} Let $M\subset S^3\subset S^4$ be the standard M\"{o}bius band. The boundary of $M$ can be pushed radially into the upper hemisphere of $S^4$ where it bounds a unique disc $D^2$ up to isotopy. The union of the M\"{o}bius band and the disc $D^2$ is an embedded, unknotted projective plane $P^2$ in $S^4$. Depending on the sign of the half-twist of the M\"{o}bius band, there are two non-isotopic unknotted projective planes, $P^2_+$ and $P^2_-$. These can be detected by one of  two invariants:  the normal Euler number, or the Brown invariant, see section~\ref{s:EuBr}.  
\end{remark}

\begin{remark}\label{rem:Eu}  If the Euler characteristic of $S$ is odd, then we may choose the projective plane $P^2$ in $S$ so that the surface $S'$ in the splitting (\ref{eq:first}) is orientable. When the Euler characteristic of $S$ is even, we may split off two unknotted projective planes leaving an orientable surface. If the Euler characteristic is less than one, we may split off an unknotted projective plane so that the resulting surface $S'$ is still non-orientable. It also follows that   there is a diffeomorphism of pairs
\[
    (X, S)\cong (X, S^2)\,\,\#\,\, k(S^4, P^2_+)\,\,\#\,\, \ell(S^4, P^2_-)
\]
where $S^2$ is a $2$-sphere embedded into $X$, and $k+\ell$ is the cross-cap number of $S$.
\end{remark}

As a consequence of  the splitting theorem (Theorem~\ref{th:1a}) we show that a version of the recent Gabai light bulb theorem (Theorem~\ref{th:5}) holds true for non-orientable surfaces as well.  

\begin{theorem}\label{th:6b} Let $X$ be an orientable $4$-manifold such that $\pi_1(X)$ has no $2$-torsion. Let $S_0$ and $S_1$ be 
two homotopic embedded $G$-inessential closed surfaces with common geometric dual $G$.  Suppose that the normal Euler numbers of $S_0$ and $S_1$ agree.  Suppose that the surface $S_0$ agrees in a neighborhood of $S_0\cap G$ with $S_1$. 
Then the surfaces are ambiently isotopic via an isotopy that fixes the geometric dual pointwise.

\end{theorem}

In the orientable case considered by Gabai  the normal Euler number does not play a role. However, this invariant is critical when the surfaces are non-orientable, see Remark~\ref{rk:newd}. Theorem~\ref{th:6b} has a number of applications. 

One source of applications is to stabilization of smoothly knotted surfaces embedded in $4$-manifolds. If $(X,S)$ is a pair consisting of a $4$-manifold  and an embedded surface, one has four types of stabilization: external stabilization $(X,S)\# (S^2\times S^2,\emptyset)$, pairwise stabilization  $(X,S)\# (S^2\times S^2,\{\text{pt}\}\times S^2)$, internal stabilization $(X,S)\# (S^4,T^2)$ and non-orientable internal stabilization  $(X,S)\# (S^4,P^2)$.   Baykur and Sunukjian proved that a sufficient number of internal stabilizations results in isotopic surfaces \cite{baykur-sunukjian:surfaces}.  In \cite{Ka}, S.~Kamada shows that  two embedded surfaces  become isotopic after enough internal non-orientable stabilizations.  

We say that an embedded surface $S$ in a $4$-manifold $X$ is \emph{$\pi_1$-trivial} if the homomorphism $\pi_1S\to \pi_1X$ induced by the inclusion $S\subset X$ is trivial. 
The hypothesis that a $\pi_1$-trivial surface $S$ is $G$-inessential for a geometric dual $G$ is always satisfied after taking the connected sum of the pair $(X, S)$ with $(S^2\times S^2, \{*\}\times S^2)$ which immediately implies the following:

\begin{corollary} Let $X$ be an orientable $4$-manifold such that $\pi_1(X)$ has no $2$-torsion. Then regularly homotopic,  $\pi_1$-trivial,  embedded surfaces in $X$ become isotopic after {\it just one} pairwise stabilization.
\end{corollary}

Quinn \cite{quinn} and Perron \cite{perron1,perron2} show that topologically isotopic surfaces in simply-connected $4$-manifolds
 become isotopic after sufficiently many external stabilizations. Auckly, Kim, Melvin, Ruberman, and Schwartz proved that \emph{just one} external stabilization was enough for ordinary topologically isotopic orientable surfaces \cite{AKMRS}.
A second consequence of the splitting theorem is a non-orientable version of this ``one is enough" theorem. 

\begin{theorem}\label{th:8a}  Let $S_0$ and $S_1$ be regularly homotopic (possibly non-orientable), embedded surfaces in an orientable $4$-manifold $X$, each with simply-connected complement. 
If the homology class $[S_0]=[S_1]$  is ordinary, then $S_0$ is isotopic to $S_1$ in $X\# (S^2\times S^2)$. If the homology class is characteristic, then the surfaces are isotopic in $X\# (S^2\tilde\times S^2)$.    
\end{theorem}

We note that by Theorem~\ref{l:Gaba} below homotopic embeddings of a closed connected orientable surface into an orientable 4-manifold are always regularly homotopic. If the surface is non-orientable, then homotopic embeddings are regularly homotopic provided that the normal Euler numbers of the embeddings agree, see Theorem~\ref{l:Gaba}. 

In \S\ref{s:EuBr} we review the notion of an unknotted projected plane as well as the definition of the normal Euler number. In \S\ref{s:sth} we prove the splitting theorem (Theorem~\ref{th:1a}). The hypothesis in Theorem~\ref{th:1a} that there exists a geometric dual $G$ is essential. In \S\ref{s:nu} we give examples of surfaces with no geometric duals that do not admit splittings. In general, the isotopy class of the surface $S\#R$ in the pair $(X, S)\#(Y, R)$ may change when $S$ and $R$ are changed by isotopy. In contrast, in \S\ref{s:nu} we show that the isotopy class of $S\#R$ is well defined when $Y$ is a sphere, see Lemma~\ref{welld}. Lemma~\ref{welld} is essential for the proof of the Gabai theorem for non-orientable surfaces (Theorem~\ref{th:6b}). Another preliminary statement is proved in \S\ref{s:rhs} where we show that homotopic surfaces with the same normal Euler numbers are regularly homotopic. Theorem~\ref{th:6b} is proved in \S\ref{s:Gabai}.  Finally, Theorem~\ref{th:8a} is proved in \S\ref{s:7}. 

Throughout the paper we work in smooth category. All $4$-manifolds are orientable but not necessarily closed. Unless stated otherwise, all surfaces in $4$-manifolds are embedded and connected. 

We thank Rob Schneiderman for generously sharing his detailed thoughts on earlier drafts of this paper, including proposing Examples~\ref{ex:5rob} and \ref{ex:6rob}, and providing Figures~\ref{W-D-ints-and-W-move}, \ref{W-D-ints-and-finger-moves},  \ref{single-finger-move-w-disk-intersects-splitting-disk},  \ref{two-finger-moves-and-W-disks-intersect-splitting-disk}, and \ref{push-self-int-across-splitting-disk}.  We are also grateful to Victor Turchin and referees for helpful suggestions and comments.

\section{Background}\label{s:EuBr}

A \emph{pair} of manifolds $(X, S)$ is a manifold $X$ together with an embedded submanifold $S$. The connected sum of pairs \cite{Ko} is denoted by
\[   
      (X_0, S_0)\,\#\, (X_1, S_1) = (X_0\,\#\, X_1, S_0\,\#\, S_1).
\]

Given a possibly non-orientable surface $S$ embedded in an oriented $4$-manifold $X$, one may define the normal Euler number $e(S)$. Take a small generic displacement $\tilde{S}$ of $S$ in the normal directions and count the algebraic number of intersection points in $S\cap \tilde {S}$. The sign of an intersection point $p$ is positive (respective, negative) if $(e_1, e_2, \tilde{e}_1, \tilde{e}_2)$ is positively (respectively, negatively)  oriented, where $e_1, e_2$ is an arbitrary basis of the tangent space $T_pS$ and $\tilde{e}_1$ and $\tilde{e}_2$ the image of $e_1$ and $e_2$ in $T_p\tilde{S}$.  
\begin{remark}\label{rk:newd}
The normal Euler number is well defined up to regular homotopy, i.e., homotopy through immersions.
\end{remark}

\begin{figure}[h]
\centering
\begin{minipage}{.5\textwidth}
  \centering
\includegraphics[width=\textwidth]{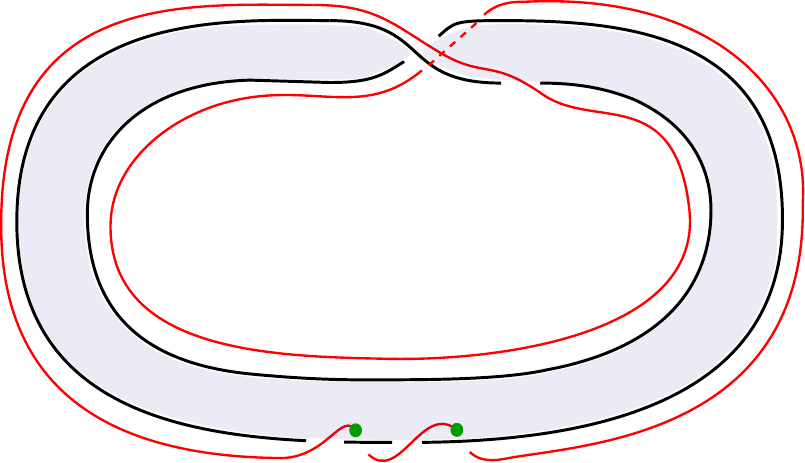}
\caption{The positive unknotted projective plane $P_+^2$ with push-off.}
\label{fig:proj}
\end{minipage}%
\begin{minipage}{.5\textwidth}
  \centering
  \includegraphics[width=0.8\textwidth]{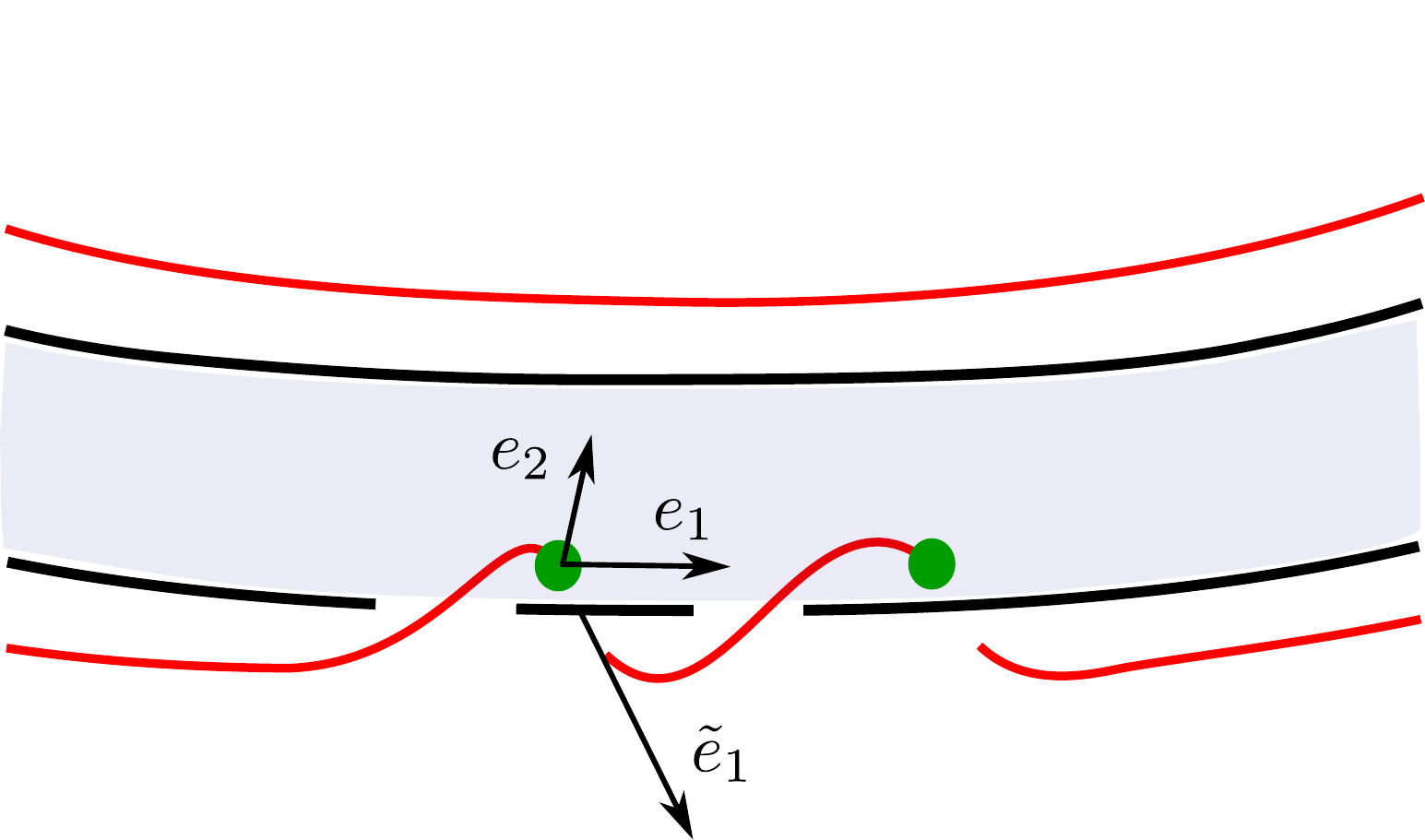}
\caption{The normal Euler number of $P_+^2$ is $-2$. Not shown is the vector $\tilde{e}_2$ directed into the interior of the lower half space $\R^4_-$.}  \label{fig:3cd}
\end{minipage}

\end{figure}

The \emph{positive unknotted} projective plane $P_+^2$   in $\R^4\subset S^4$ is obtained by capping off the gray right-handed M\"{o}bius band in Figure~\ref{fig:proj} with a disc $D$  in the upper half space $\R^4_+= [0,\infty)\times\R^3$. There is a displacement $\tilde D$ of $D$ in  $\R^4_+$ that has an empty intersection with $D$. It is bounded by the red curve in Figure~\ref{fig:proj}.  This curve has zero linking number with the boundary of the M\"{o}bius band. The red curve may be extended  to the lower half space $\R^4_-$
and then capped with a red M\"obius band to obtain a displacement $\tilde{P}_+^2$. The only points of intersection are the two green points in Figure~\ref{fig:proj} and \ref{fig:3cd}.  Orienting the tangent space of the grey M\"{o}bius band at a green point by vectors $e_1$ and $e_2$ and taking  displaced vectors  $\tilde{e}_1$ and $\tilde{e}_2$ in the tangent space of the red M\"{o}bius band, we can see that the two intersection points are counted negatively.
Therefore, the normal Euler number of $P_+^2$ is negative two. 
The negative unknotted projective plane $P_-^2$ is obtained by capping off a left-handed M\"{o}bius band with a disc in the upper half space $\R^4_+$. Its normal Euler number is $2$.

 \begin{remark}   Let $r$ be the linear transformation of $\R^4$ given by $(t, x, y, z)\mapsto (-t, -x, y, z)$. It takes the projective plane obtained by capping off the right-handed M\"{o}bius band with a disk in the upper half-space to the projective plane obtained by capping off a left-handed M\"{o}bius band with a disc in the lower half space $\R^4_-$.  
\end{remark}

\begin{remark}
There is a different invariant defined when $S$ is a characteristic (possibly non-orientable) closed surface embedded in an orientable  $4$-manifold $X$. It is called the Brown invariant \cite{Br}. The Brown invariant will not play a role in this paper. 
\end{remark}

\section{The splitting theorem} \label{s:sth}
In this section we prove the main splitting theorem. We begin with a simple observation that any splitting is determined by a special disk.

\begin{lemma}\label{l:1} Let $S$ be a closed surface embedded into a possibly non-orientable $4$-manifold $X$. Suppose that there is a closed $4$-disc $U\subset X$ such that the intersection $U\cap S$ is a M\"{o}bius band $M$ and $\partial U$ is 
an embedded submanifold of $X$ intersecting $S$ transversally. Suppose that $\partial M$ is an unknot in $\partial U\cong S^3$. Then $(X, S)$ is diffeomorphic to the connected sum of pairs of manifolds $(X, S')$ and $(S^4, P^2_\pm)$ where $S'$ is a surface obtained from $S\setminus \mathop\mathrm{Int}(U)$ by capping off its only boundary component with a $2$-disc, where $ \mathop\mathrm{Int}(U)$ is the interior of $U$.   
\end{lemma}
\begin{proof} Since $\partial M$ is an unknot in $\partial U$, the boundary of each of the pairs $(\overline{X\setminus U}, \overline{S\setminus U})$ and $(U, U\cap S)$
is diffeomorphic to the boundary of the pair $(D^4, D^2)$ of standard discs. In other words, the boundary of each of the two pairs can be capped off by the pair of standard discs to produce pairs $(X, S')$ and $(S^4, P^2_\pm)$ whose connected sum is diffeomorphic to $(X, S)$. 
\end{proof}

 In practice, the closed disc $U$ in Lemma~\ref{l:1} is constructed by taking a closed regular neighborhood of a $2$-disc $D$ such that $\partial D$ is the central closed curve of the M\"{o}bius band $M$, the interior of $D$ does not contain points of $S$ and $D$ is nowhere tangent to $S$. If such a disc $D$ exists, then we say that $D$ is the \emph{core of the splitting} of Lemma~\ref{l:1}.

Suppose a connected surface $S$ possesses a geometric dual $G$. Given another surface $R\subset X$, an intersection point $p\in S\cap R$ can be \emph{tubed off} using $G$ along a path $\gamma$ in $S$ from the point $p$ to $\hat G\cap S$, where $\hat G$ is a parallel copy of $G$, see Figure~\ref{fig:2}.   The result of this procedure is a new surface $\hat S$ obtained from $R$ by taking the union of $R$ and a copy  $\hat G$, removing a disc neighborhood $D_G$ of $\hat G\cap S$ in $\hat G$, removing a disc neighborhood $D_S$ of $p$ in $R$,  attaching a tube $S^1\times [0,1]$ along $\gamma$  to the two new boundary components of $R\setminus D_S$ and $\hat G\setminus D_G$, and smoothing the corners.    

\begin{figure}[h]
\centering
\begin{minipage}{.5\textwidth}
  \centering
\includegraphics[width=0.6\textwidth]{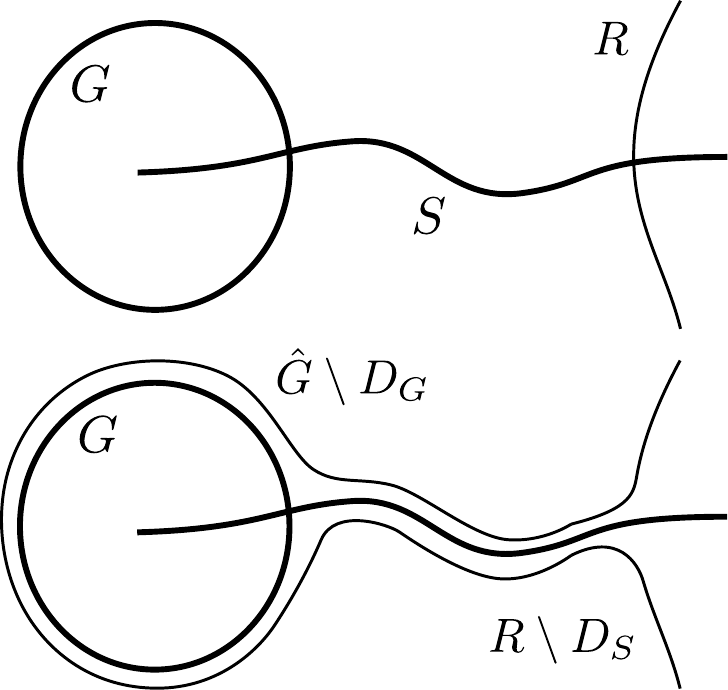}
\caption{Using a geometric dual to tube off an intersection point.}  \label{fig:2}
\end{minipage}%
\begin{minipage}{.5\textwidth}
  \centering
  \includegraphics[width=0.55\textwidth]{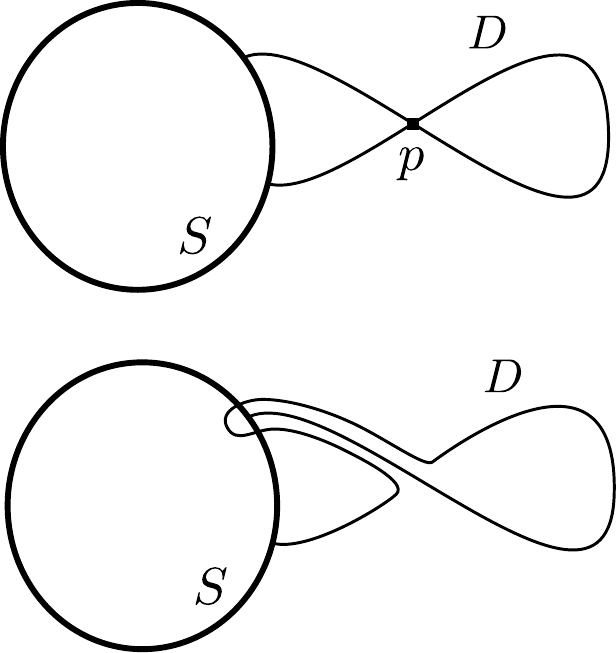}
\caption{Using a finger move to remove an intersection point.}  \label{fig:3c}
\end{minipage}

\end{figure}

\begin{proof}[Proof of Theorem~\ref{th:1a}]  Let $\alpha$ be a simple closed orientation reversing curve in $S$ that is disjoint from $G$. 
Since $S$ is $G$-inessential, the curve $\alpha$ bounds an immersed disc $D$ in $X\setminus G$. 

Using finger movers, boundary twists, and tubing with the geometric dual, we may convert $D$ into the core of a splitting. Indeed,  we may join any self intersection point $p$ of $D$ with a point in $\partial D$ by a  curve and  use a finger move to eliminate the self intersection point $p$ of $D$, see Figure~\ref{fig:3c}. By repeating this procedure we obtain an embedded disc $D\subset X\setminus G$ that may intersect $S$ in interior points.  Since $S\setminus \alpha$ is path connected, for each point $p$ in $D\cap S$ there is a path from $p$ to the unique intersection point $G\cap S$. Thus, we may use the geometric dual $G$ to tube off the intersection points of $D$ with $S$, see Figure~\ref{fig:2}. Thus there exists an embedded $2$-disc $D$ in $X\setminus G$ such that the intersection $D\cap S$ is the curve $\alpha$. Furthermore, we may assume that $D$ approaches $S$ orthogonally (with respect to a Riemannian metric on $X$).

\begin{figure}[h]
\centering
\begin{minipage}{.5\textwidth}
  \centering
\includegraphics[width=0.6\textwidth]{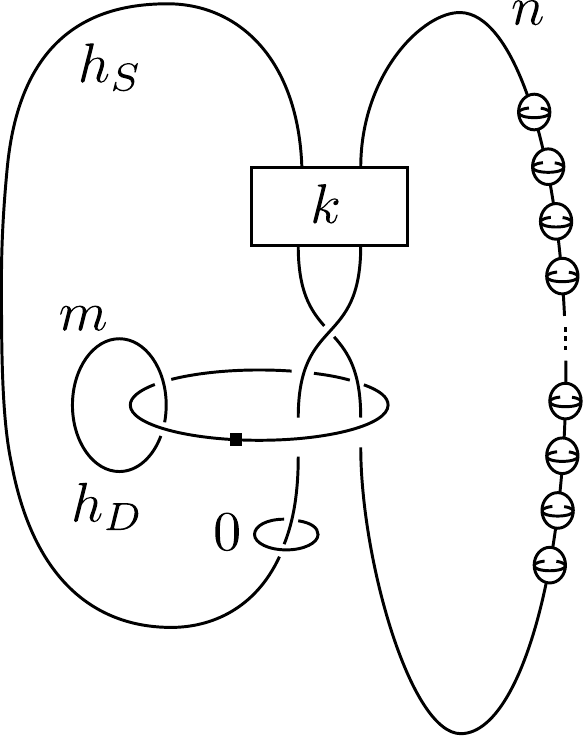}
\caption{A neighborhood of the surface $S$, dual disc $G$ and an embedded disc $D$. }\label{fig:6}
\end{minipage}%
\begin{minipage}{.5\textwidth}
  \centering
  \includegraphics[width=0.6\textwidth]{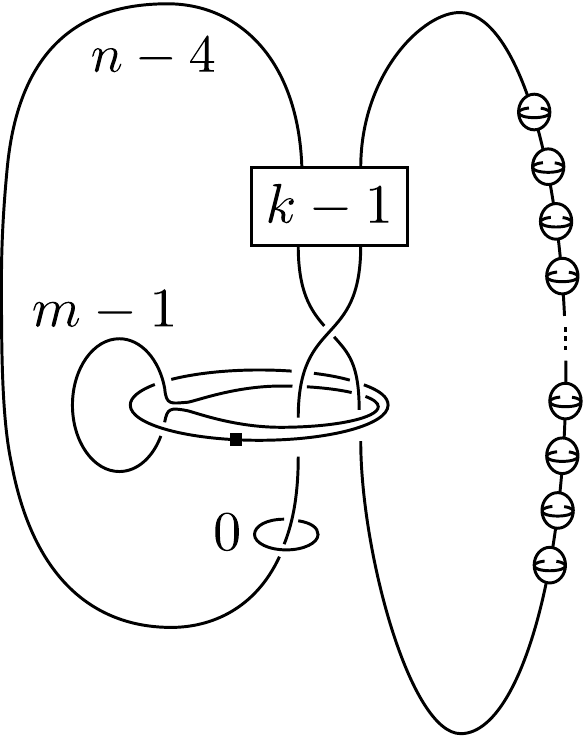}
\caption{A twist of the $1$-handle}\label{fig:4}
\end{minipage}

\end{figure}

 In what follows we will use standard techniques of Kirby calculus including handle notation, handle slides, and handle twisting. We refer the reader to an excellent exposition by Gompf-Stipsicz, see \cite[Chapters 4,5]{GS}. 
A neighborhood $S^1\times D^3$ of $\alpha$ in $X$ is diffeomorphic to the complement in $D^4$ of a neighborhood of $D^2$. 
 This is depicted by a dotted circle representing the boundary of the disc $D^2$, see Figure~\ref{fig:6}. Such a neighborhood  of $\alpha$ already contains the M\"{o}bius band neighborhood of $\alpha$ in $S$. With respect to a trivialization of $S^1\times D^2$, it twists $k+1/2$ times for some integer $k$. The rest of a regular neighborhood of $S$ in $X$ is obtained from the described neighborhood of $\alpha$ by attaching $1$-handles that correspond to thickenings of $1$-handles of $S$ and one $2$-handle $h_S$ which corresponds to the thickening of the $2$-cell in a cell decomposition of $S$ corresponding to a perfect Morse function. A regular neighborhood of a geometric dual contributes a $2$-handle attached along a meridian of $S$ with zero framing. A regular neighborhood of the disc $D$ also contributes a $2$-handle $h_D$ attached along a circle that passes over the $1$-handle $h_1$ once and which, \emph{a priori}, could be linked with the attaching circle of $h_S$.  Using the geometric dual $G$, the attaching circle of $h_D$ can be unlinked from the attaching circle of $h_S$. Furthermore, since the attaching circle of $h_D$ is isotopic to $\alpha$, it is unknotted.  We denote by $m$ and $n$ the framings of the attaching circles of the $2$-handles $h_D$ and $h_S$ respectively. 

\begin{figure}[h]
  \centering
  \includegraphics[width=0.2\textwidth]{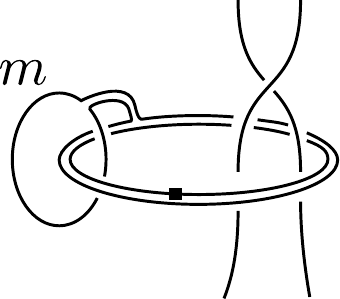}
  \caption{Sliding the $2$-handle $h_D$ over the $1$-handle links the attaching sphere $h_D$ with the attaching sphere $h_S$ and changes the framing by $\pm 2$.}  \label{fig:3}
\end{figure}

Giving one of the attaching discs of the $1$-handle a full left rotation results in linking the attaching circles of $h_D$ with $h_S$ as well as decreasing simultaneously $m$ and $k$ by $1$ and $n$ by  $4$, see Figure~\ref{fig:4}. (A right rotation has the opposite effect.) Thus, we may assume that $k$ is $0$ or $-1$, corresponding to $\pm 1/2$ twists, and $m$ is even. If $k=0$ (respectively, $k=-1$) and $m$ is odd, then a full $-1$ rotation (respectively, $+1$ rotation) of the $1$-handle results in even $m$ and $k=-1$ (respectively, $0$).

Sliding the attaching circle of $h_D$ along $h_1$, links the attaching circles of $h_D$ and $h_S$ and changes the framing $m$ by $\pm 2$, see Figure~\ref{fig:3}. In view of the geometric dual, we may again unlink the attaching circles $h_D$ from $h_S$. 
To summarize, we may assume that in Figure~\ref{fig:6}, the twisting number $k$ is $0$ or $-1$, and the framing $m$ is $0$. 

Let $U$ denote a closed regular neighborhood of the disc $D$; then $U$ itself is a closed $4$-disc. Without loss of generality, we may assume that the attaching sphere for $h_S$ is in $\partial U$. Since $m=0$ and $k$ is $0$ or $-1$, it follows that the attaching sphere for $h_S$ is an unknot.  
Theorem~\ref{th:1a} now follows from Lemma~\ref{l:1}.

\end{proof}

\section{Problems with splitting and sums}\label{s:nu}

\subsection{Non-existence and non-uniqueness of splittings}

In the absence of a geometric dual the splitting surgery along arbitrary $1$-sided curves may not be possible, see Example~\ref{ex:2}.

\begin{example}\label{ex:2} According to the Massey theorem~\cite{Ma69}, there exists an embedding of a closed non-orientable surface $S\subset S^4$ of Euler characteristic $\chi$ with normal Euler number $\nu=2\chi-4, 2\chi, ..., 4-2\chi$. We may choose an embedding so that $\nu\ne \pm 2$. Suppose that there exists a splitting surgery representing $(S^4, S)$ by a connected sum of $(S^4, S')$ and $(S^4, P^2_\pm)$ where $S'$ is a closed orientable surface. Since the normal Euler number of $P^2_\pm$ is $\mp 2$ and the normal Euler number of a closed orientable surface embedded  in $S^4$ is trivial, we deduce that the normal Euler number of their connected sum $S$ is $\pm 2$, which contradicts the assumption that $\nu \ne \pm 2$. Therefore, such a splitting surgery does not exist. 
\end{example}

In fact it may be the case that no splitting is possible as in the following example.
\begin{example}\label{ex:3}
The rational elliptic surface has a cusp fiber $F$. In Kirby calculus, a neighborhood of this fiber is obtained by attaching a $0$-framed  $2$-handle to $D^4$ along a right-handed trefoil. The right handed trefoil bounds an obvious  M\"{o}bius band. Capping the M\"{o}bius band with the core of the $2$-handle results in an embedded $P^2$ representing the fiber class $[F]$. Notice that this class is characteristic. A splitting of the form
\[
(E(1),P^2) \cong (E(1),S^2)\# (S^4,P^2)
\]
would imply the existence of a smoothly embedded sphere representing the fiber class $[F]$ in $E(1)$ contradicting the Kervaire-Milnor theorem~\cite{KM}.
\end{example}

When splitting is  possible, it need not be unique. Indeed changing the homotopy class of the core of the splitting can change the integral homology class of the summands. In Example~\ref{ex:nonun2} we describe a general way to get inequivalent splittings of a map.

\begin{example}\label{ex:nonun2}
Suppose that $D$ is a core of a splitting of  a non-orientable surface $S$ in an orientable $4$-manifold $X$. Suppose that the splitting results in the decomposition
\[
    (X, S) \cong (X, S_D') \# (S^4, P_\pm),
\]
where $S_D'$ is an orientable surface.
Let $A$ be a homotopically non-trivial embedded sphere in $X\setminus (S\cup D)$ with trivial normal bundle. Consider the splitting with the core $D\# A$,
\[
    (X, S) \cong (X, S_{D\# A}') \# (S^4, P_\pm).
\] 
Then $[S_{D\# A}']=[S_D']+2[A]$, where the orientation of $A$ agrees with the orientation of $D\# A$.  
\end{example}

 In the last example,  the $\Z_2$ homology classes of the surfaces $S_{D\# A}'$ and $S_D'$ in the decomposition agree.  
However,  the integral homology classes of the surfaces   are not the same.  Thus, Example~\ref{ex:nonun2} shows that the connected sum  decomposition of pairs is not unique.  This should not be surprising as the connected sum decomposition of manifolds  is not unique in $4$-dimensions. We now give one more example where by changing the core of the splitting we are able to change from splitting off a copy of $P_+$ to splitting off a copy of $P_-$.

\begin{example}\label{ex:nonun}  Let $X$ denote the manifold $\C P^2\#(S^2\times S^2)$ and $S$ the submanifold $S^2\times \{0\}$ in the second factor $S^2\times S^2$ of $X$. 
The projective plane in  $(X, S)\# (S^4,P_+)$ has geometric dual $\{*\}\times S^2$ in the connected factor $S^2\times S^2$ of $X$. We claim that we can split off either $P_+$ or $P_-$. Indeed, the $P_+$-splitting is obvious. To describe the $P_-$-splitting, let $D$ denote the core of the  $P_+$-splitting. We note that this corresponds to $k=0$ and $m=0$ in Figure~\ref{fig:6} using the trivialization of $D$.  Replacing $D$ with its connected sum with $\C P^1\subset \C P^2$, results in a model corresponding to $k=0$ and $m=1$, see Figure~\ref{fig:11} where the $2$-handle corresponding to $\C P^1$ is denoted by $H$. To view the neighborhood in  the trivialization of the new disk we apply a twist to the $1$-handle. This model  
 corresponds to $k=-1$ and $m=0$, see Figure~\ref{fig:13}. Now we may slide the handle $D+H$ along the $2$-handle $-G$ twice to obtain the core for a $P_-$-summand. 
In other words, 
\[
    (X,S)\# (S^4,P_+) \cong (X,S')\# (S^4,P_-).
\]
The normal Euler number of $P_-$ is $2$, while the normal Euler number of $P_+$ is $-2$. Consequently, the normal Euler number of $S'$ is $-4$. In fact, $[S'] =  [S]+2[\C P^1]-4[\{*\}\times S^2]$  in the homology group  
$H_2(\C P^2\#S^2\times S^2; \Z)$.
\end{example}

\begin{figure}[h]
\centering
\begin{minipage}{.5\textwidth}
  \centering
\includegraphics[width=0.6\textwidth]{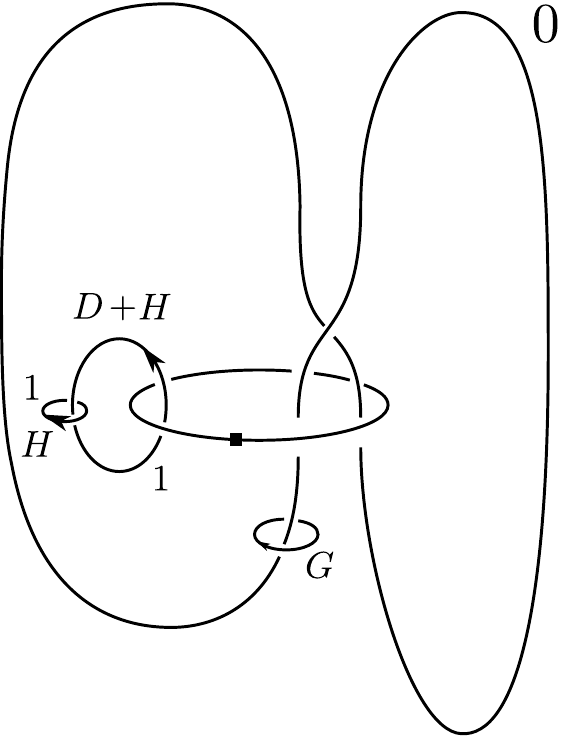}
\caption{The handle $D+u$ of the connected sum of the disc $D$ with $\C P^1$.}\label{fig:11}
\end{minipage}%
\begin{minipage}{.5\textwidth}
  \centering
  \includegraphics[width=0.6\textwidth]{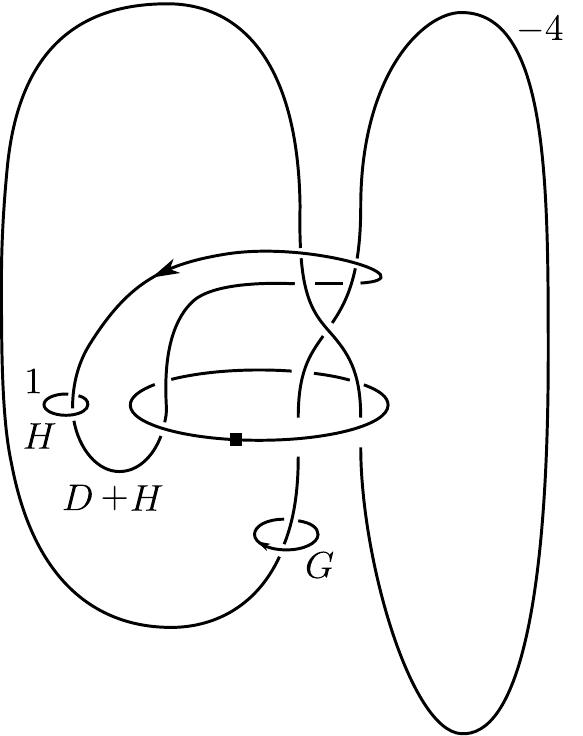}
\caption{A twist of the $1$-handle}\label{fig:13}
\end{minipage}
\end{figure}

The connected sum of two manifolds $X$ and $Y$ is defined by removing coordinate open balls $D_X\subset X$ and $D_Y\subset Y$, and then identifying the new boundaries in $X\setminus D_X$ and $Y\setminus D_Y$ appropriately.

\subsection{Connected sums of pairs}

In general, connected sums of pairs are not well-defined up to isotopy. However, connected sums are well-defined under hypothesis of Lemma~\ref{welld}. In this subsection we will prove Lemma~\ref{welld}; it will be used in the proof of Theorem~\ref{th:6b}. 

The connected sum of pairs of manifolds is defined similarly by means of pairs of balls. We say that $(D_X, D_S)\subset (X, S)$ is a \emph{pair of coordinate open balls} if $(\bar{D}_X,\bar{D}_S)$ is an embedded pair of closed discs, and $\partial D_X$ intersects $S$ along $\partial D_S$, and the pair $(D_X, D_S)$ is parametrized by a diffeomorphism from the standard pair of unit balls in $(\R^4, \R^2)$. The connected sum of pairs $(X, S)$ and $(Y, R)$ is defined by removing pairs of coordinate open discs $(D_X, D_S)\subset (X, S)$ and $(D_Y, D_R)\subset (Y, R)$ and then identifying the new boundaries appropriately.  

 In this paper we are interested in internal sums, which are defined by taking the connected sum with $(Y, R)$ where $R$ is a surface embedded in $Y=S^4$. In this case, without loss of generality we may assume that $D_Y$ is the lower hemisphere, and identify $Y\setminus D_Y$ with a closed coordinate ball. Then $X\#S^4$ is canonically diffeomorphic to the original manifold $X$ as the connected sum operation replaces the coordinate open ball $D_X$ with the interior of the coordinate ball $S^4\setminus D_Y$. We will write $S\#_i R$ for the resulting surface in $X=X\# S^4$ when  $i\co D_X\to X$ is a specified inclusion and $D_Y$ is an open lower hemisphere in $Y=S^4$. 

To motivate Lemma~\ref{welld}, we note that in the case of connected sums of pairs of manifolds there is an additional subtlety. Namely, 
let $S_0$ and $S_1$ be two isotopic surfaces in $X$ that agree in a coordinate open ball $D_X$. Furthermore, suppose that $(X, S_0)$ and $(X, S_1)$ share the same pair of coordinate balls $(D_X, D_S)$. Then
the ambient space $X\#Y$ in the pair $(X, S_0)\# (Y, R)$
coincides with the ambient space in the pair $(X, S_1)\# (Y, R)$. 
However,  in general, the surface $S_0\# R$ may not be isotopic to $S_1\# R$ in $X\# Y$, as the isotopy of $X\setminus \{0\}$, where $\{0\}$ is the center of the coordinate ball $D_X$,  may not admit an extension to an isotopy of $X\#Y$.

Given a pair $(X, S)$, we say that a tuple of vectors $v_1,..., v_4$ at a point $x\in S$ is an \emph{adapted} frame if it is a basis for the tangent space $T_xX$ and if the vectors $v_1$ and $v_2$ form a basis for the tangent space $T_xS$. If $X$ and $S$ are oriented, then we additionally require that the basis $\{v_1,..., v_4\}$ for $T_xX$ and $\{v_1, v_2\}$ for $T_xS$ to be positively oriented. 
We note that up to isotopy the pair of the coordinate open balls $(D_X, D_S)$ in $(X, S)$ determines and is determined by the standard coordinate adapted frame 
$e_1,..., e_4$ in $T_{0}D_X\subset TX$.

\begin{lemma}\label{welld} 
Let $S_0$ and $S_1$ be connected isotopic surfaces in an oriented $4$-manifold $X$, and $R$ be a surface embedded in $S^4$. 
Let $i_0$ and $i_1$ be possibly different orientation preserving embeddings of an open coordinate $4$-ball $D^4$ into $X$ such that $(i_kD^4, i_kD^2)$ is a pair of coordinate open discs in $(X, S_k)$ for $k=0,1$. In the case where $S_0$ and $S_1$ are oriented, suppose that the isotopy from $S_0$ to $S_1$ is orientation preserving, and the frames associated with the pairs $(i_kD^4, i_kD^2)$ of coordinate balls are adapted.  
 Then the surface $S_0\#_{i_0} R$ is isotopic to the surface $S_1\#_{i_1} R$ in $X$. 
\end{lemma}
\begin{proof}
The ambient isotopy taking $S_0$ to $S_1$ takes the connected sum $S_0\#_{i_0}R$ to $S_1\#_{j}R$   for some embedding $j\co D^4\to X$. Let $\{v_i\}$ and $\{w_i\}$ denote the frames over $(X, S_1)$ corresponding to 
the embeddings $j$ and $i_1$ respectively. Since the frames are adapted, by applying an ambient isotopy of $X$ that fixes $S_1$ setwise, we may assume that $v_1=w_1$ and $v_2=w_2$. Since both $\{v_i\}$ and $\{w_i\}$ are positively oriented frames over $X$, the coincidences $v_1=w_1$ and $v_2=w_2$ imply that
there is an ambient isotopy of $X$ that fixes $S_1$ setwise and takes the frame $\{v_i\}$ to the frame $\{w_i\}$. Thus, the surface $S_0\#_{i_0} R$ is isotopic to the surface $S_1\# R$. 
This completes the proof of Lemma~\ref{welld}. 
\end{proof}

\section{Regularly homotopic surfaces}\label{s:rhs}

At one step in the Gabai's proof of the $4$-dimensional light bulb theorem \cite{Ga} one modifies a given homotopy between orientable surfaces into a regular homotopy. In this section we will show that
the hypotheses on the surfaces in the non-orientable version of the Gabai theorem (Theorem~\ref{th:6b})  guarantee that the (possibly non-orientable) surfaces are still regularly homotopic. 

By the Smale-Hirsch theorem~\cite{Hi59}, the space  of immersions of a manifold $S$ into a manifold $X$ is weakly homotopy equivalent to the space $\Imm^F(S, X)$ of smooth injective bundle homomorphisms $TS\to TX$ provided  that $\dim S< \dim X$ or that $S$ is open. Let $f\co S\to X$ denote a smooth map, and $\mathop\mathrm{Hom}(TS, f^*TX)\to S$ the standard $\mathop\mathrm{Hom}$-bundle with a fiber over $s\in S$ given by homomorphisms $T_sS\to T_{f(s)}X$. Let $V(TS, f^*TX)\to S$ denote the subbundle of the  $\mathop\mathrm{Hom}$-bundle of injective homomorphisms. Then the space of injective bundle homomorphisms $TS\to TX$ covering a smooth map $f\co S\to X$ can be identified with the space of sections of $V(TX, f^*TX)\to S$. In fact,
there is a natural fibration
\[
 \Imm^F(S, X) \to C^\infty(S,X),
\]
where $C^\infty(S, X)$ is the space of smooth maps $f\co S\to X$, with  fiber over  $f$ given by the space
 of sections of the bundle $V(TS,f^*TX)\to S$. 

When the dimension of $S$ is $2$ and the dimension of $X$ is $4$, the fiber of the fiber bundle $V(TS,f^*TX)$ over $S$ is homotopy equivalent to $O(4)$. 
The case where $S$ is a sphere is established by Smale~\cite{Sm}, and it is used in \cite{Ga}. Since we need the non-orientable case, we give a quick outline of its proof here.

\begin{theorem}\label{l:Gaba} Suppose that $f$ and $g$ are two homotopic embeddings of a closed connected surface $S$ into an orientable $4$-manifold $X$. If the surface $S$ is non-orientable, suppose, in addition, that the normal Euler numbers of $f$ and $g$ agree. Then the embeddings $f$ and $g$ are regularly homotopic. 
\end{theorem}  
\begin{proof}

Choose a handle decomposition of $S$ with a unique $2$-cell. By a general position argument, we may assume that a homotopy of $f$ to $g$ restricts to an isotopy of a regular neighborhood of the $1$-skeleton of $S$.   Furthermore, the isotopy of this neighborhood extends to an isotopy of the ambient manifold $X$. Thus, we may assume  that $f$ and $g$  agree in the neighborhood of the $1$-skeleton of $S$. Consequently, the normal bundles $N(f)$  and $N(g)$ agree over the same neighborhood.

Since  $f$ and $g$ are homotopic, their normal Euler numbers agree when $S$ is orientable. In particular, in both cases, when $S$ is orientable or non-orientable, under the hypotheses of Theorem~\ref{l:Gaba},  the normal bundle of the immersion $f$ is isomorphic to that of $g$. Even more is true, the isomorphism already given over the $1$-skeleton  extends to an isomorphism $N(f)\cong N(g)$. Let $\hat f$ denote the inclusion of $N=N(f)$ into $X$, and $\hat g\co N\to X$ denote the composition of the isomorphism $N\cong N(g)$ and the inclusion. 
By the homotopy lifting property, a homotopy $\hat f_t$ of $\hat f$ to $\hat g$ extends to a homotopy of $d\hat f$ through injective bundle homomorphisms $TN\to \hat f_t^*TX$. We may assume that the resulting homomorphism $h\co TN\to \hat g^*TX$ agrees with $d\hat g$ over the $1$-skeleton of $S\subset N$. Since the fiber of the bundle $V(TN, \hat g^*TX)\to TN$ is homotopy equivalent to $O(4)$ and $\pi_2(O_4)=0$, the homomorphism $h$ can be further deformed so that it agrees with $\hat g$ over $S\subset N$, see \cite[\S 29]{St} Finally, since $S$ is a deformation retract of $N$, the homomorphism $h$ is homotopic to $\hat g$ through injective bundle homomorphisms. 
It follows that $\hat f$ is regularly homotopic to $\hat g$, hence $f$ is regularly homotopic to $g$.
\end{proof}

\begin{remark}\label{rk:1s}
The argument in the proof of Theorem~\ref{l:Gaba} shows that one may assume that the regular homotopy in the conclusion of Theorem~\ref{l:Gaba} restricts to an isotopy away from any prescribed disk in the surface $S$. 
\end{remark}

\begin{remark} There are many knotted projective planes in $S^4$, e.g, see \cite{Vi}. On the other hand, all projective spaces embedded in $S^4$ are null homotopic, and have normal Euler number $\pm 2$, see \cite{Ma69}. Therefore, 
by Theorem~\ref{l:Gaba} all projective planes embedded in $S^4$ are regularly homotopic to $P_+$ or $P_-$.   
\end{remark}

\section{The Gabai light-bulb theorem for non-orientable surfaces}\label{s:Gabai}

Recently Gabai proved the following theorem, see \cite[Theorem 9.7]{Ga}.  

\begin{theorem}[Gabai, \cite{Ga}]\label{th:5}  Let $X$ be an orientable $4$-manifold such that $\pi_1(X)$ has no $2$-torsion. Two homotopic embedded $G$-inessential orientable surfaces $S_0$ and $S_1$  with common geometric dual $G$ are ambiently isotopic via an isotopy that fixes the geometric dual pointwise. 
\end{theorem}

In this section
we prove Theorem~\ref{th:6b}, which asserts that the Gabai result is still true for non-orientable surfaces as well, provided that the normal Euler numbers of the surfaces agree. 
\begin{remark}
If a surface $S$ has a geometric dual, then $S$ is ordinary. Thus the Brown invariant does not play a role in this theorem.
\end{remark}
The idea of the proof is to use the splitting theorem  (Theorem~\ref{th:1a}) to reduce the general case of possibly non-orientable surfaces $S_0$ and $S_1$ to the case of orientable surfaces by representing $S_0$ and $S_1$ as internal connected sums of orientable surfaces $S_0'$ and $S_1'$ with unknotted projective planes. As examples in \S\ref{s:nu} show, in general the surfaces $S_0'$ and $S_1'$ may not even be homotopic. Lemma \ref{ag1sk}  below shows that we may assume that the surfaces agree away from an open ball. For surfaces meeting the conclusion of Lemma \ref{ag1sk},  we prove Lemma \ref{nhs1} and Lemma~\ref{nhs} ensuring that there is a splitting such that the surfaces $S_0'$ and $S_1'$ are homotopic.

\begin{lemma}\label{ag1sk}
Let $S_0$ and $S_1$ be regularly homotopic surfaces embedded in a $4$-manifold $X$. Then there exist an embedded surface $S_2$ in $X$ and a regular neighborhood $U$ of a point in $S_0$ such that $S_0$ agrees with $S_2$ on the complement of $U$, and $S_2$ is regularly homotopic to $S_0$ by a homotopy that is constant on the complement of $U$, and $S_1$ is isotopic to $S_2$.
\end{lemma}

\begin{proof}
Let $R:I\times S_0\to X$ be a regular homotopy with $R_0(S_0)=S_0$ and $R_1(S_0)=S_1$. By Remark~\ref{rk:1s}, we may assume that $R$ restricts to an isotopy in a neighborhood of a $1$-skeleton of $S_0$ with complement an open disk $U$. By the isotopy extension theorem, there is an ambient isotopy $J:I\times X\to X$ that agrees with $R$ restricted to $S_0\backslash U$. Clearly, the surface $S_2:=J_1^{-1}(S_1)$ is isotopic to $S_1$, and $S_0$ agrees with $S_2$ in the complement to $U$. The required regular homotopy of $S_0$ to $S_2$ fixing the complement of $U$ is given by
$\hat R_t(x):=J_t^{-1}\circ R_t(x)$. 
\end{proof} 

Lemma \ref{ag1sk} is the first step in the proof of Theorem~\ref{th:6b}.  It establishes an isotopy of $S_1$ to $S_2$ so that $S_0$ and $S_2$ agree away from a regular neighborhood of a point. Thus, in the proof of Theorem~\ref{th:6b}, without loss of generality, we may assume that $S_1$ agrees with $S_0$ away from a neighborhood of a point. Lemmas~\ref{nhs1} and \ref{nhs} will establish that there are cores of splittings such that the 
 surfaces $S_0'$ and $S_1'$ obtained by splitting off unknotted projective planes from $S_0$ and $S_1$ respectively are still homotopic.

\begin{lemma}\label{nhs1} Let $X$ be an orientable $4$-manifold. 
Let $S_0\subset X$ be an embedded connected $G$-inessential non-orientable surface for some geometric dual $G$ for $S_0$. Let $D_0$ be a core of a splitting such that $S_0\setminus \partial D_0$ is non-orientable.    Let $S_1$ be a $G$-inessential surface that agrees with $S_0$ away from a disc neighborhood $U\subset S_1$ of a point such that $\partial D_0\cap U=\emptyset$ and $U\cap G=\emptyset$. Suppose that $S_1\setminus G$ is regularly homotopic to $S_0\setminus G$ relative to the complement to $U$ in $X\setminus G$.  
Then there exists  a core $D_1$ of a splitting for $S_1$ that agrees with $D_0$ in a neighborhood of $\partial D_0$  such that the surfaces $S_0'$ and $S_1'$ obtained by splitting are regularly homotopic relative to $G$.
\end{lemma}
\begin{proof}  If the intersection of the interior of $D_0$ and  $S_1$ is empty, then $D_1=D_0$ is a  core of a splitting for $S_1$ with the desired property. In the rest of the argument we will assume that $D_0\cap S_1$ consists of a single point as the general case is similar. 
\begin{figure}[h]
\includegraphics[width=0.8\textwidth]{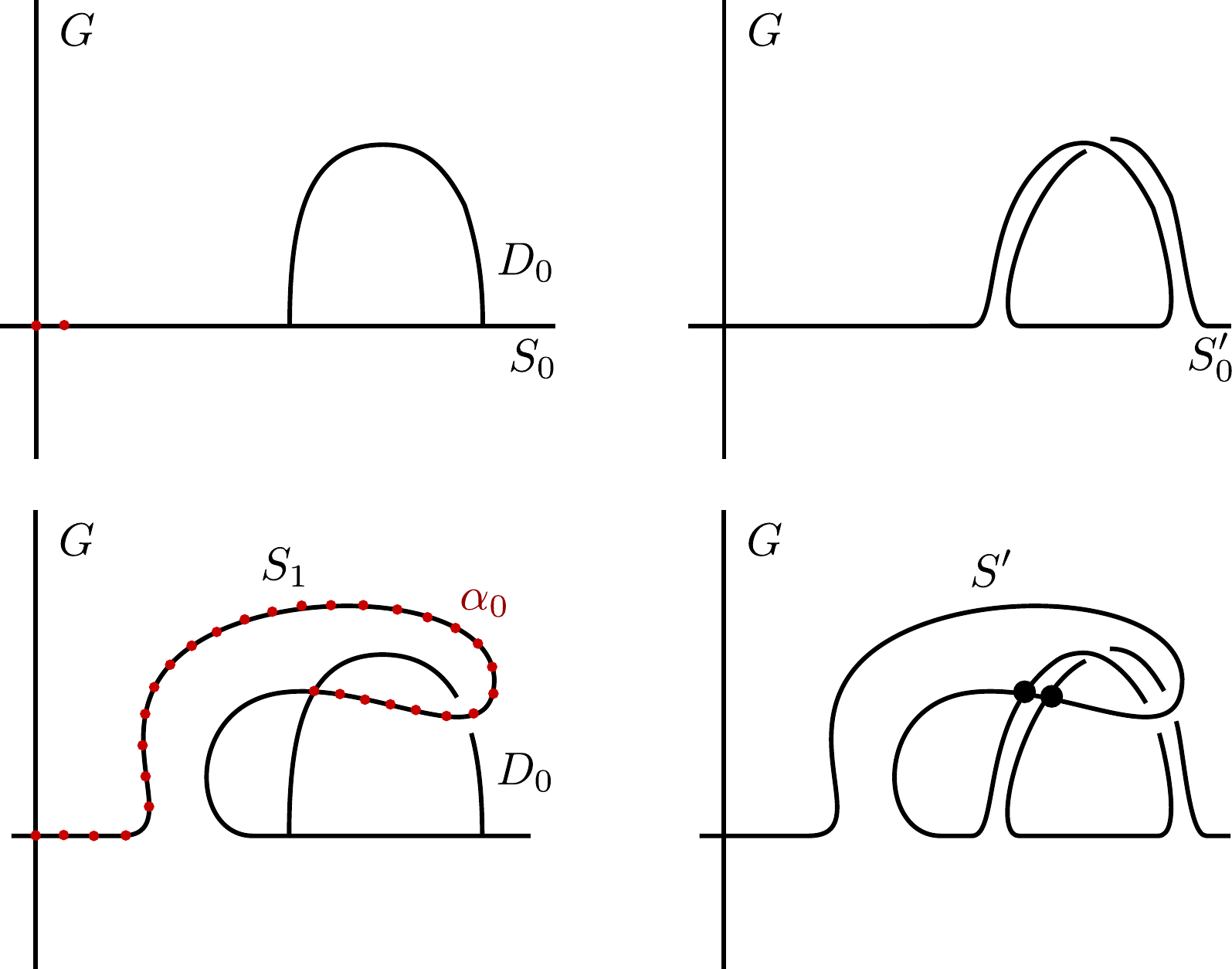}
\caption{Each intersection point in $S_1\cap D_0$ corresponds to two intersection points in $S'\cap S_0'$.}\label{fig:dr11}
\end{figure}
There exists a curve $\alpha_0$ on $S_1$ from the intersection point $D_0\cap S_1$ to a parallel copy of $G$. Let $D_1$ denote the disc $D_0$ with the intersection $D_0\cap S_1$ tubed along $\alpha_0$, see Fig.~\ref{fig:2}. We claim that the splitting $S_0'$ of $S_0$ along $D_0$ is regularly homotopic relative to $G$ to the splitting $S_1'$ of $S_1$ along $D_1$. Indeed, the surface $S_1'$ can be constructed differently. To begin with we may split the surface $S_1$ along $D_0$ to obtain an immersed intermediate surface $S'$ with two self-intersection points (see Fig.~\ref{fig:dr11}), and then we may tube the two self-intersection points of $S'$ along  
respectively two curves $\alpha_+$ and $\alpha_-$ which coincide with $\alpha_0$ except near the end points. In particular, the surface $S_1'$ is obtained by removing from $S'$ two discs $B_\pm$ near the two self-intersection points, and then attaching two discs $D_\pm$ in a neighborhood of $\alpha_\pm\cup G$, see Fig.~\ref{fig:dr10}. 
\begin{figure}[h]
\includegraphics[width=0.45\textwidth]{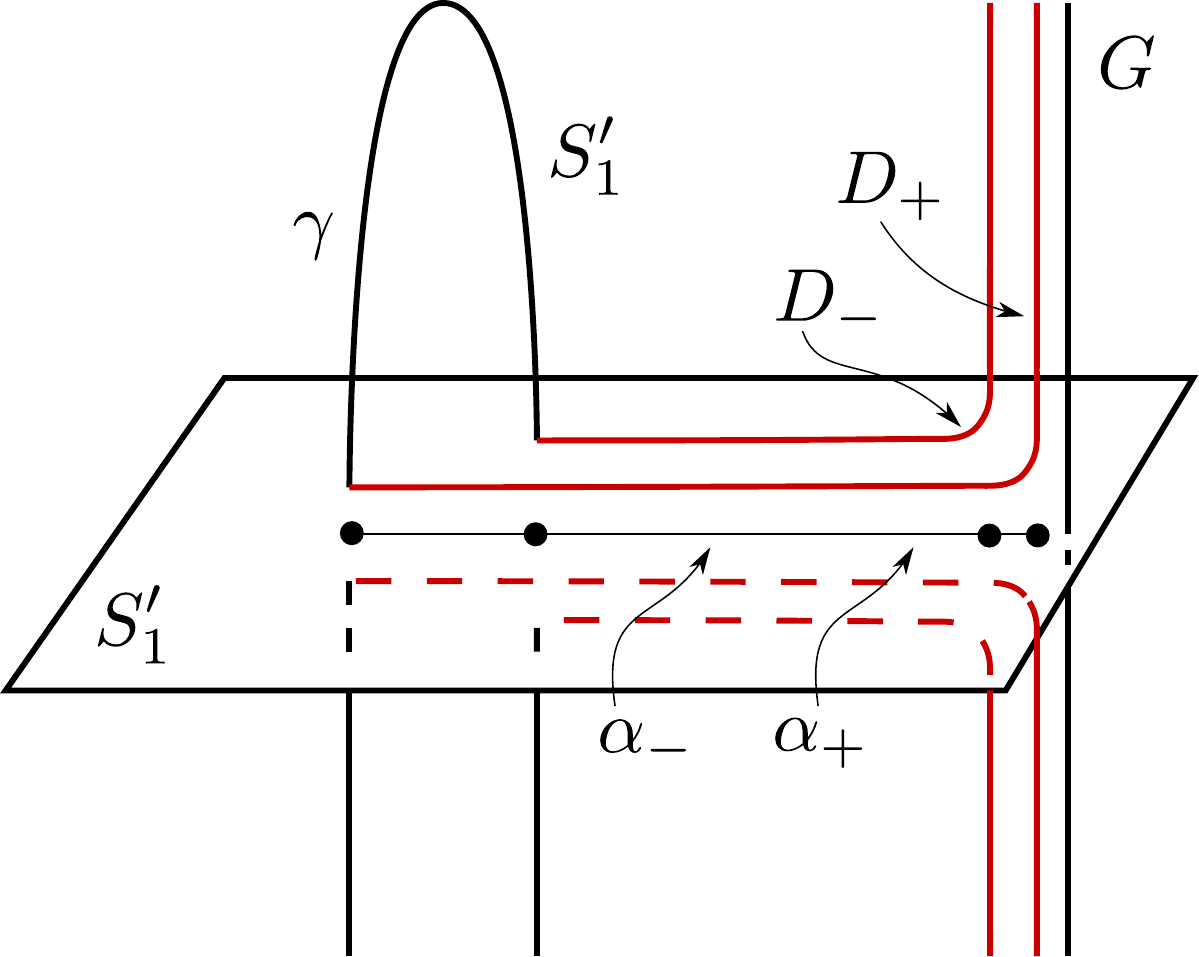}
\caption{Tubing the self-intersection points of $S'$.}\label{fig:dr10}
\end{figure}

Since $S_1$ is regularly homotopic to $S_0$ relative to the complement to $U$, their splittings $S'$ and $S_0'$ along $D_0$ are regularly homotopic. The normal Euler numbers of $S'$ and $S_1'$ are the same since the Euler number of $G$ is trivial. By Theorem~\ref{l:Gaba} it remains to show that $S'$ is homotopic to $S_1'$.

The disc $D_+$ can be obtained from $D_-$ by a $C^0$-small isotopy. Therefore,  an orientation on $D_+$ gives rise to an orientation on $D_-$.  For a curve $\gamma$ in $S_1'$ from $\partial D_+$ to $\partial D_-$, let $N(\gamma)$ denote a closed neighborhood of $\gamma$ in $S_1'$ such that the unions  $D_+\cup D_- \cup N(\gamma)$ and $B_+\cup B_-\cup N(\gamma)$ are boundary connected sums of discs along $\gamma$. 
Since $S_1'$ is non-orientable, we may choose $\gamma$ so that the compatible orientations on $D_+$ and $D_-$ do not extend to an orientation of the disc $D_+\cup D_- \cup N(\gamma)$. 
We denote by $C$ the union of the two discs $D_+\cup D_- \cup N(\gamma)$ and $B_+\cup B_-\cup N(\gamma)$.  We will show that the sphere $C$ is null-homotopic by showing that a lift of $C$ to the universal covering $\widetilde{X\setminus G}$ of $X\setminus G$ is null-homotopic.

Choose a reference lift of a neighborhood $N(D_0)$ of $D_0$, and take lifts of $B_\pm$ contained in the lift of $N(D_0)$ and lifts of $D_\pm$ meeting these lifts of $B_\pm$. Since there is an isotopy taking $D_-$ to $D_+$ keeping the boundary in $N(D_0)$, we see that $D_\pm$ are obtained from $B_\pm$ by tubing with parallel copies of the same lift of a copy of $G$. 
Denote these lifts by $\widetilde{D_\pm}$, $\widetilde{B_\pm}$, and $\widetilde G$. Since $S_1$ is $G$-inessential, $S'$ is $G$-inessential, so the loop obtained by joining the endpoints of $\gamma$ in $N(D_0)$ is null-homotopic and we see that the lift $\widetilde C$ of $C$ containing $\widetilde{B_\pm}$ also contains $\widetilde{D_\pm}$, so in the homology of $\widetilde{X\setminus G}$ we have $[\widetilde C] = [\widetilde G] - [\widetilde G] = 0$. It follows that $C$ is null-homotopic, establishing that $S'$ (and therefore $S'_0$) is homotopic to $S_1'$.

\end{proof}

Up to isotopy (rel boundary), a generic regular homotopy between properly embedded surfaces in a $4$-manifold can be expressed as finitely many finger moves followed by the same number of Whitney moves,
where an \emph{isotopy of an immersion} is the conjugation of the map by diffeotopies of range and domain, e.g., see \cite[\S 1.6]{FQ}, 
\cite[\S 3]{PRT} and 
\cite[Lemma 8]{BT}.
By the Isotopy Extension Theorem an isotopy of an embedding extends to an ambient isotopy, but in general isotopies of immersions are ``non-ambient''.

The following lemma will be used to keep track of intersections between a surface and a splitting disk that are created during a regular homotopy. 

\begin{lemma}\label{lem:control-htpy}
Let $S_0$ be a possibly non-orientable surface properly embedded in a $4$-manifold, and let $D$ be a $2$-disk with $\partial D$ embedded in $S_0$ such that the interior of $D$ is disjoint from $S_0$. If $S_0$ is homotopic (rel boundary) to an embedded surface $S_1$ by a regular homotopy $S_t$, 
then $S_1$ is ambiently isotopic (rel boundary) to the result of the following moves starting with $S_0$:
\begin{enumerate}
\item\label{item:lem-finger-moves-away-from-D}
finger moves supported away from $D$,

\item\label{item:lem-finger-moves-into-D}
finger moves pushing sheets of $S_t\setminus \partial D$ into the interior of $D$, 

\item\label{item:lem-finger-moves-across-boundary-D}
internal finger moves pushing sheets of $S_t\setminus \partial D$ across $\partial D$,

\item\label{item:lem-w-moves-away-from-D}
Whitney moves guided by Whitney disks which are disjoint from $D$.

\end{enumerate}
Each finger move in item~(\ref{item:lem-finger-moves-into-D}) creates a pair of interior intersections between $S_t$ and $D$, and each internal finger move in item~(\ref{item:lem-finger-moves-across-boundary-D}) creates a single interior intersection between $S_t$ and $D$. Here an ``internal finger move'' is guided by an arc contained in $S_t$ that starts at a self intersection of $S_t$ and ends at a point in $\partial D$ (Figure~\ref{W-D-ints-and-finger-moves}).

\end{lemma}

\begin{proof}
By a small perturbation (rel boundary) we may assume that the regular homotopy $S_t$ is generic, and we fix a description of $S_t$ up to isotopy as finger moves followed by Whitney moves.  
The finger moves are supported near arcs which can be assumed to be pairwise disjoint, with each finger move creating two self intersections in the ``middle level'' of the homotopy which is the result of doing all the finger moves. The Whitney moves are supported near pairwise disjointly embedded framed Whitney disks pairing the self intersections of the middle level, with each Whitney move eliminating two self intersections. Each Whitney disk has a neighborhood diffeomorphic to the left side of Figure~\ref{W-D-ints-and-W-move}, with the result of the Whitney move shown in the right side of Figure~\ref{W-D-ints-and-W-move}.

Since they are supported near arcs, the finger moves starting with $S_0$ as in item~(\ref{item:lem-finger-moves-away-from-D}) may be assumed to be supported away from $D$ by general position.
 
The Whitney disks that describe the rest of the regular homotopy will generically have interior and boundary intersections with $D$; see the left side of Figure~\ref{W-D-ints-and-W-move} for an example. These intersections will lead to interior intersections between $S_1$ and $D$, as shown in the right side of  Figure~\ref{W-D-ints-and-W-move}. 

The right side of Figure~\ref{W-D-ints-and-finger-moves} shows how to change a Whitney disk $W$ to a ``smaller'' Whitney disk $W'$ which is disjoint from $D$ by applying finger moves as in items~(\ref{item:lem-finger-moves-into-D}) and~(\ref{item:lem-finger-moves-across-boundary-D}). The combination of these finger moves followed by a Whitney move guided by $W'$ is ambient isotopic to the result of doing a Whitney move guided by $W$. 

The proof is completed by applying the modifications of Figure~\ref{W-D-ints-and-finger-moves} to all the Whitney disks on the middle level, and then doing the Whitney moves on the resulting smaller Whitney disks. The resulting embedded surface is isotopic to, and hence ambiently isotopic to, $S_1$.
\end{proof}

\begin{figure}[ht!]
         \centerline{\includegraphics[scale=.5]{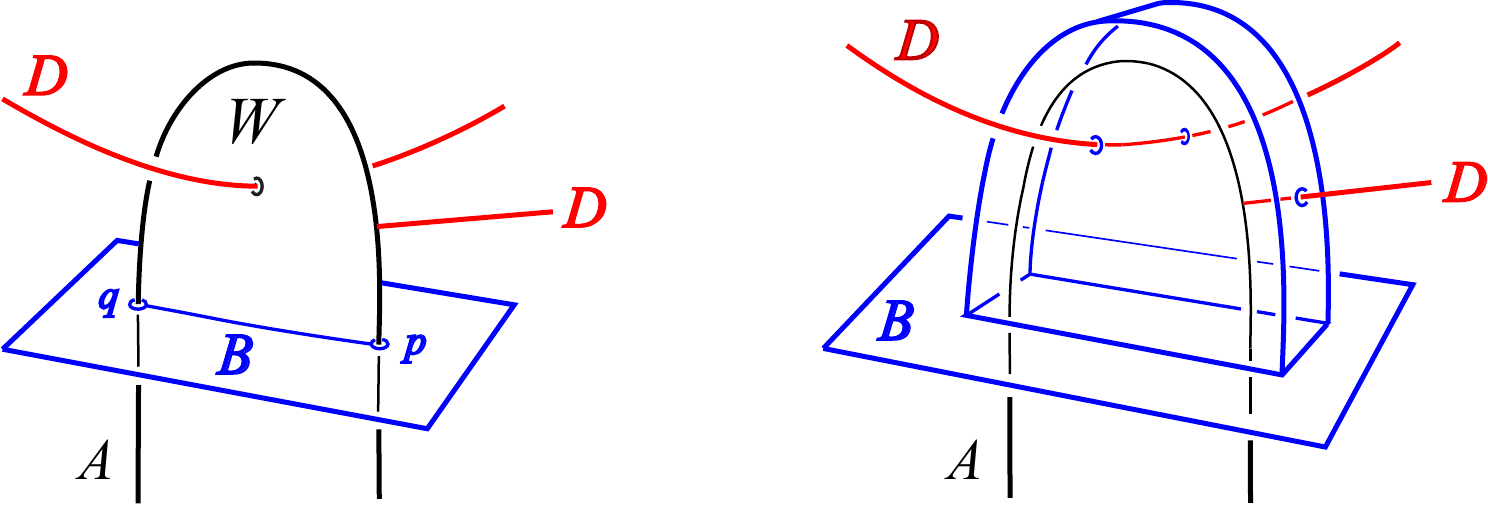}}
         \caption{Left: A neighborhood $B^4=B^3\times I$ of a framed embedded Whitney disk $W$ pairing intersections $\{p,q\}= A\pitchfork B$ between sheets $A$ and $B$ of $S_t$. Here $W$ and $B$ are contained in the `present' while $A$ and $D$ are understood to extend into the `past' and `future' $I$-directions. Right: The Whitney move on $W$ eliminates $p$ and $q$ by moving $B$ across $W$.  Each interior intersection in $W\pitchfork D$ yields two interior intersections between $B$ and $D$, and each intersection in $\partial W\cap\partial D$ yields one interior intersection between $B$ and $D$.}
         \label{W-D-ints-and-W-move}
\end{figure}

\begin{figure}[h!]
         \centerline{\includegraphics[scale=.5]{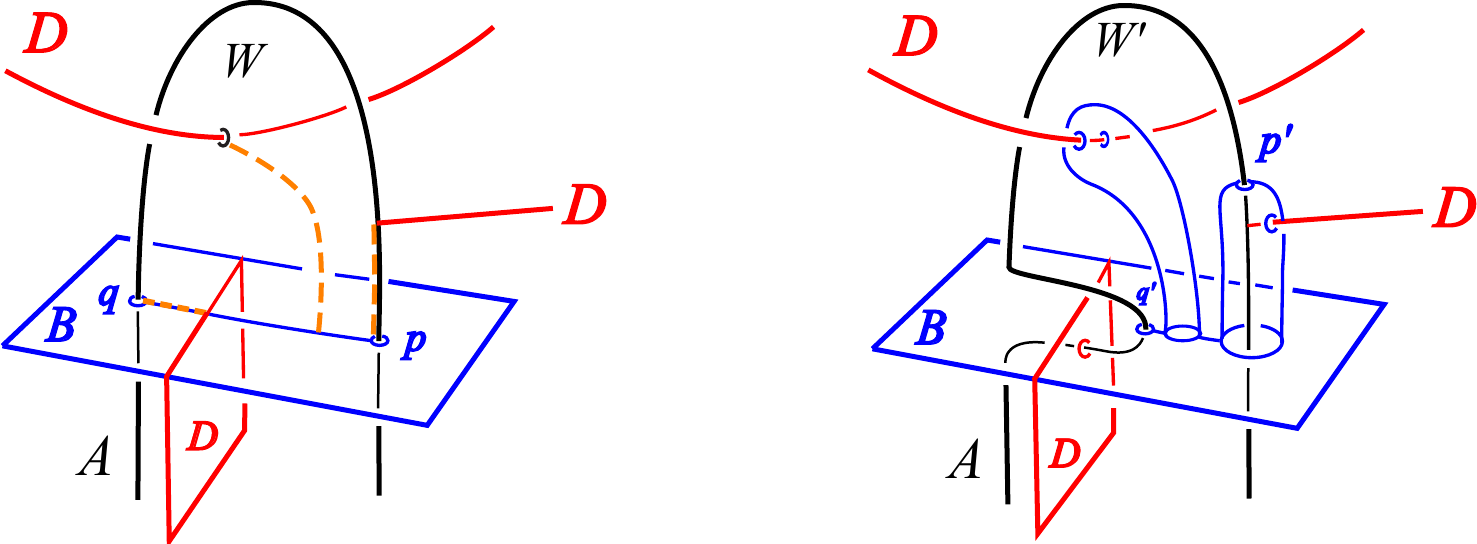}}
         \caption{The dashed orange arcs on the left guide finger moves, yielding $W'$ pairing $q'$ and $p'$ on the right. If $W$ had more interior intersections with $D$, then more finger moves would be done along pairwise disjoint arcs in the interior of $W$. And if $\partial W$ had more intersections with $\partial D$, then the guiding arcs for the internal finger moves along $\partial W$ would be extended across all such intersections (which is the same as iterating the indicated internal finger moves across single boundary intersections).}
         \label{W-D-ints-and-finger-moves}
\end{figure}

The proof of Lemma~\ref{nhs} below also requires a simple case of the self-intersection number of a generic immersed, oriented surface. If $j:R\hookrightarrow X$
is such a surface (generally denoted by $R$), the \emph{double point set} is $D(j):=\{p\in J(R)\,|\, \# j^{-1}(p) >1\}$ and the \emph{self-intersection number} is $R\cdot R := \sum_{D(j)} \epsilon(p)$ where $\epsilon(p)=\pm 1$ is determined by the orientations in the standard way.

\begin{remark}  Let $S_t$ be a non-orientable surface in an orientable $4$-manifold $X$, and $D$ a core of a splitting such that $S_t\setminus \partial D$ is orientable. Let $W$ be a Whitney disc defining a Whitney move on the surface $S_t$. The boundary of the Whitney disc $W$ consists of two arcs $\partial_1W$ and $\partial_2W$ with common endpoints $p$ and $q$. 
Then  the signs of $p$ and $q$ in the definition of the self-intersection number $(S_t\setminus \partial D)\cdot (S_t\setminus \partial D)$ are in general distinct from the relative signs of the double points $p$ and $q$ with respect to the Whitney disc sign convention. Indeed, Fig.~\ref{fig:18} can be used to describe one such example. Suppose that the arc $\partial_1W$ intersects $\partial D$ transversally at a unique point. Then there is a curve $\gamma$ (shown in red in the figure) on $S_t\setminus \partial D$  joining the two endpoints $p$ and $q$ of the arc $\partial_1W$. The relative signs of $p$ and $q$  with respect to the Whitney disc sign convention come from orientations of regular neighborhoods of $\partial_1W$ and $\partial_2W$. The existence of the Whitney framing implies that the points $p$ and $q$ have opposite relative signs. 
 The signs of $p$ and $q$ in the definition of
$(S_t \cap \partial D) \cdot (S_t \cap \partial D)$ are determined by means of regular neighborhoods of $\partial_2W$ and $\gamma$. Since the union
of  $\gamma$ and $\partial_1W$ is an orientation reversing curve on $S_t$, these two sign conventions cannot match.
\end{remark}

\begin{figure}[h]
  \centering
  \includegraphics[width=0.7\textwidth]{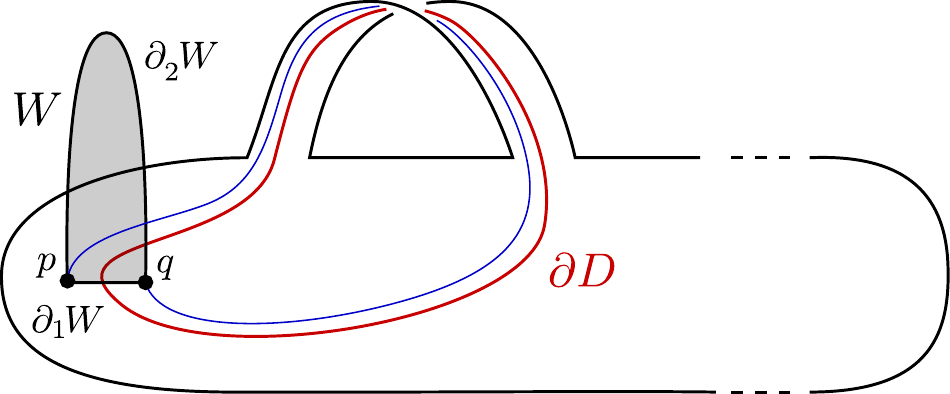}
  \caption{the signs in the definition of the self-intersection number $(S_t\setminus \partial D)\cdot (S_t\setminus \partial D)$ are in general distinct from the relative signs of the double points in the description of a Whitney disk.}  \label{fig:18}
\end{figure}

\begin{remark}\label{rem:10c}
Let $S_0$ be an embedded non-orientable surface in an orientable $4$-manifold $X$ with a core $D_0$ of a splitting such that $S_0\setminus D_0$ is oriented. Then the disc $D_0$ inherits an orientation. Indeed, the splitting of $S_0$ along $D_0$ is obtained by removing a M\"obius band neighborhood of $\partial D_0$ in $S_0$ and attaching a disc $B$ along $D_0$. We orient the resulting surface $S_0'$ so that over $S_0'\setminus B$  the orientation of $S_0'$ agrees with the orientation on $S_0\setminus D_0$. On the other hand, without loss of generality we may assume that near $\partial D_0$ the projection of the tubular neighborhood of $D_0$ onto $D_0$ defines a double cover of a neighborhood of $\partial B$ in $B$ over a neighborhood of $\partial D_0$ in $D_0$. In particular, the disc $D_0$ inherits an orientation.  
\end{remark}

\subsection*{Examples of homotopies with boundary intersections between Whitney disks and spitting disks.}
These examples illustrate how a homotopy can create intersections with a splitting disk coming from boundary intersections with Whitney disks, and also illustrate some of the subtleties regarding the ``opposite sign condition'' for Whitney disks [Freedman--Quinn, Sec.1.4] in the setting of non-orientable surfaces.

Figure~\ref{single-finger-move-w-disk-intersects-splitting-disk} and Figure~\ref{two-finger-moves-and-W-disks-intersect-splitting-disk} describe two regular homotopies from a $G$-inessential embedding $S_0:\R P^2 \hra X^4$ to $G$-inessential embeddings $S_1$ and $S_1'$ that have interior intersections with a splitting disk $D$ on $S_0$. Both homotopies fix $S_0$ outside the subdisk bounded by the dashed grey circle, and in particular fix neighborhoods of $\partial D$. The preimage of $\partial D$ is shown in red in both figures.
It is assumed that $S_0$ is $G$-inessential, with the geometric dual $G$ suppressed from the figures and assumed to intersect $S_0$ outside the indicated subdisk. 
Fix an orientation of $S_0\setminus \partial D$.

\begin{figure}[h!]
         \centerline{\includegraphics[scale=.45]{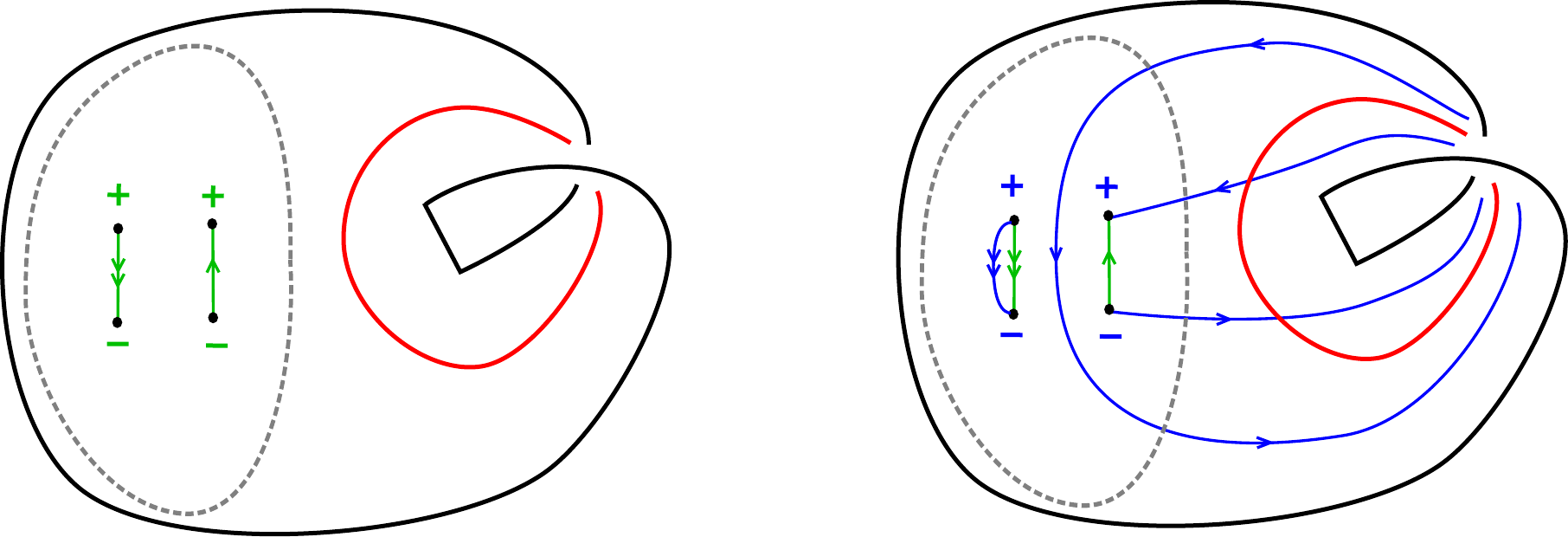}}
         \caption{Left: The preimage (green) of the boundary of a Whitney disk on $S_{0.5}$ that is `inverse' to a finger move on $S_0$. Right: The preimage (blue) of the boundary of a Whitney disk that leads to $S_1$.}
         \label{single-finger-move-w-disk-intersects-splitting-disk}
\end{figure}

\begin{example} \label{ex:5rob} The left side of Figure~\ref{single-finger-move-w-disk-intersects-splitting-disk} shows the preimage of $S_{0.5}:\R P^2\imra X$ which is the result of performing a single finger move on $S_0$. The green arcs are the preimage of the boundary of a Whitney disk `inverse' to the finger move. Horizontally adjacent pairs of black dots map to the same self intersection of $S_{0.5}$, with the lower pair mapping to the negative self intersection, and the upper pair mapping to the positive self intersection.

The right side of Figure~\ref{single-finger-move-w-disk-intersects-splitting-disk} also shows the preimage of $S_{0.5}$, but now with the blue preimage of the boundary of a new Whitney disk $W$ such that $\partial W$ intersects $\partial D$ in two points (the green arcs are just shown for reference to the left picture). The existence of a framed $W$ with the indicated boundary follows from the $G$-inessential assumption: $\partial W$ bounds an immersed disk in $X\setminus G$ which can be promoted to an embedded framed Whitney disk with interior disjoint from $S_{0.5}$ and $G$ by boundary-twisting, pushing-down, and tubing into $G$. Note that the ``opposite-signs'' condition for intersections to be paired by a framed Whitney disk is satisfied by the blue Whitney arcs since one blue Whitney arc ``goes \emph{twice} over the twisted $1$-handle'' of $S_{0.5}$.

Performing a $W$-Whitney move which moves a neighborhood in $S_{0.5}$ of the image of the short blue double-arrow boundary arc across $W$, while fixing a neighborhood of the image of the long blue single-arrow boundary arc, completes the homotopy to a $G$-inessential embedding $S_1:\R P^2\hra X$. The two intersections between $\partial W$ and $\partial D$ in $S_{0.5}$ yield two interior intersections between $S_1$ and $D$; compare Figure~\ref{W-D-ints-and-W-move}.
By the proof of Lemma~\ref{nhs} below these interior intersections have opposite signs in the complement of $\partial D$, so can be eliminated by tubing $D$ into $G$, yielding a splitting disk for $S_1$ with the same homotopy class rel boundary as $D$.
\end{example}

\begin{figure}[h!]
         \centerline{\includegraphics[scale=.5]{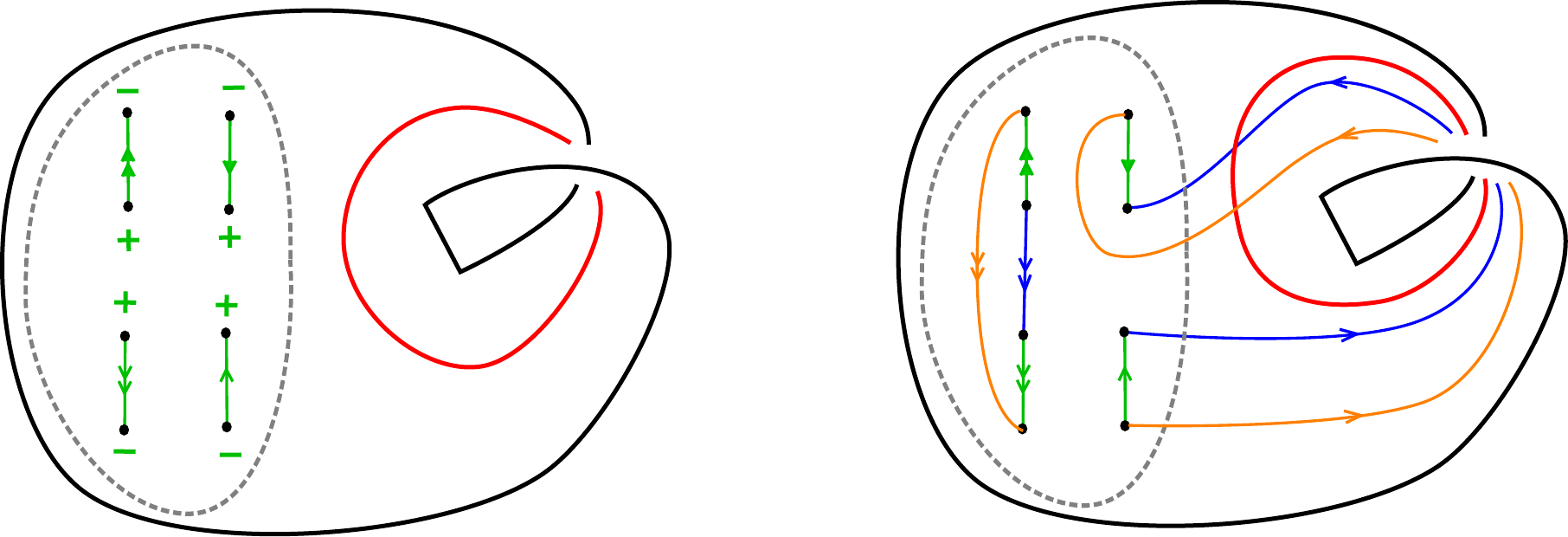}}
         \caption{Left: After two finger moves on $S_0$. Right: Before two Whitney moves lead to $S_1$.}
         \label{two-finger-moves-and-W-disks-intersect-splitting-disk}
\end{figure}

\begin{example}\label{ex:6rob} The left side of Figure~\ref{two-finger-moves-and-W-disks-intersect-splitting-disk} shows the preimage of $S'_{0.5}:\R P^2\imra X$ which is the result of performing two finger moves on $S_0$. The green arcs are the preimages of the boundaries of inverse Whitney disks to these finger moves, similarly to Figure~\ref{single-finger-move-w-disk-intersects-splitting-disk}. We assume that both finger moves are guided by arcs representing equal elements of $\pi_1X$.

The right side of Figure~\ref{two-finger-moves-and-W-disks-intersect-splitting-disk} also shows the preimage of $S'_{0.5}$, but now with the preimages in blue and orange of the boundaries of two new Whitney disks, each intersecting $\partial D$ in a single point (again the green arcs are just shown for reference to the left picture). The existence of appropriate Whitney disks with these boundaries follows from the $G$-essential assumption (similarly to the previous example), since the four self intersections determine the same element of $\pi_1X$.
The opposite-signs condition is satisfied by each of the blue and orange Whitney disk boundaries even though they are pairing
self intersections which are ``like-signed'' using the orientation in the complement of $\partial D$.
This is because one of each of the blue and orange Whitney arcs ``goes once over the twisted $1$-handle'' of $S'_{0.5}$, so the relative signs of the self intersections are switched compared with the green Whitney arcs contained in the complement of $\partial D$.

Performing the blue and orange Whitney moves which move neighborhoods in $S'_{0.5}$ of the images of the short blue and orange double-arrow boundary arcs, while fixing neighborhoods of the images of the long blue and orange single-arrow boundary arcs, completes the homotopy to a $G$-inessential embedding $S'_1:\R P^2\hra X$. Each of the two intersections between the Whitney disk boundaries and $\partial D$ in $S'_{0.5}$ yields an interior intersection between $S'_1$ and $D$, and again by
Lemma~\ref{nhs} these interior intersections have opposite signs in the complement of $\partial D$, so can be eliminated by tubing $D$ into $G$, yielding a splitting disk for $S'_1$ with the same homotopy class rel boundary as $D$.
\end{example}

\begin{lemma}\label{nhs} Let $X$ be an oriented $4$-manifold. 
Let $S_0\subset X$ be an embedded connected $G$-inessential non-orientable surface for a geometric dual $G$ for $S_0$.  Let  $D_0$ be a core of a splitting such that $S_0\setminus \partial D_0$ is orientable. Let $S_1$ be a $G$-inessential surface that agrees with $S_0$ away from a disc neighborhood $U\subset S_1$ of a point such that $\partial D_0\cap U=\emptyset$ and $G\cap U=\emptyset$. Suppose that $S_1\setminus G$ is regularly homotopic in $X\setminus G$ to $S_0\setminus G$ relative to the complement to $U$. Then there exists  a core $D_1$ of a splitting for $S_1$ that agrees with $D_0$ in a neighborhood of $\partial D_0$  such that $D_0\cup D_1$ is null homotopic and the relative normal Euler number of $(D_0, v)$ agrees with that of $(D_1, v)$ with respect to any vector field $v$ over $\partial D_0=\partial D_1$ normal to $D_0$ (and therefore to $D_1$).
\end{lemma}

Given an embedded disc $D$ in a $4$-manifold $X$ and a nowhere zero vector field $v$ over $\partial D$ normal to $D$ in $X$, the relative normal Euler number of $(D, v)$ is defined to be the algebraic number of signed intersection points of $D$ and any of its generic displacement $\tilde D$ such that $\partial \tilde D$ is obtained by displacing $\partial D$ in the direction of $v$.

\begin{proof}

First, let us consider the case where $X\setminus G$ 
is simply connected. If $S_1\cap D_0=\partial D_0$ then $D_1:=D_0$ satisfies the conclusion of the lemma. Otherwise, create $D_1$ from $D_0$ by tubing $D_0$ into parallel copies of the geometric dual $G$ to eliminate the interior intersections $S_1\pitchfork D_0$. Since $G$ is framed, the relative Euler number of $(D_1,v)$ is equal to that of $(D_0, v)$. To show that $D_0\cup D_1$ is null homotopic, it suffices to check that $S_1\pitchfork D_0$ consists of oppositely-signed pairs, where the signs come from choosing and fixing orientations on $S_1\setminus \partial D_0$ and $D_0$, so that the copies of $G$ added to $D_0$ to form $D_1$ come in oppositely oriented pairs which do not change the homotopy class rel boundary of $D_0$.

To check this we may assume that $S_1$ is equal to the result of the finger moves and Whitney moves realizing the regular homotopy $S_t$ as in Lemma~\ref{lem:control-htpy}, since an isotopy of embeddings will extend to an ambient isotopy of $S_1\cup D_1$ which preserves signed intersections.
Define for each generic $S_t$ the following combination of algebraic intersection numbers:
$$
\psi(S_t):=S_t\cdot D_0 + \frac{1}{2}S_t\cdot S_t
$$
where orientations are induced by choosing and fixing orientations on $S_0\setminus \partial D_0$ and $D_0$.

We will show that $\psi$ is invariant under the moves of Lemma~\ref{lem:control-htpy}. This implies Lemma~\ref{nhs} in the simply-connected case. Indeed, 
since $\psi(S_0)=0$, we conclude that $\psi(S_1)=0$ in view of the invariance of $\psi$.  On the other hand, since $S_1$ is embedded, it follows that  $S_1\cdot D_0=0$; i.e., ~$S_1\pitchfork D_0$ consists of oppositely-signed pairs.

Now let us verify that $\psi$ is invariant under the moves of Lemma~\ref{lem:control-htpy}:
Each finger move of item~(\ref{item:lem-finger-moves-away-from-D}) guided by an arc in the complement of $D_0$ creates only a pair of self intersections that admit a local Whitney disk 
whose boundary is disjoint from $\partial D_0$, so the intersection pair has opposite signs using the induced orientation on $S_t\setminus \partial D_0$.

Similarly, each finger move of item~(\ref{item:lem-finger-moves-into-D}) creates only a pair of oppositely signed intersections between $S_t$ and $D_0$ using induced orientation on $S_t\setminus \partial D_0$ and $D_0$.

By Remark~10, we may assume that the fixed orientation on $D_0$ has been chosen so that pushing a self intersection $p\in S_t\pitchfork S_t$ across $\partial D_0$ as in item~(\ref{item:lem-finger-moves-across-boundary-D}) of Lemma~\ref{lem:control-htpy} creates an interior intersection between $S_t$ and $D$ that has the same sign as $p$, and also converts $p$ to $p'\in S_t\pitchfork S_t$ having the opposite sign as $p$ (see Figure~\ref{push-self-int-across-splitting-disk}). (Note that it is with respect to the orientation of $S_t\setminus \partial D_0$ that $p$ and $p'$ have opposite signs, so the self intersections
$p'$ and $q'$ shown in Figure~\ref{W-D-ints-and-finger-moves} still satisfy the ``opposite signs'' criteria for admitting a Whitney disk \cite[Sec.1.4]{FQ}, no matter how many intersections with $\partial D_0$ are crossed by the internal finger moves, since the sheets $A$ and $B$ have only been changed by isotopy.)
\begin{figure}[h]
         \centerline{\includegraphics[scale=.5]{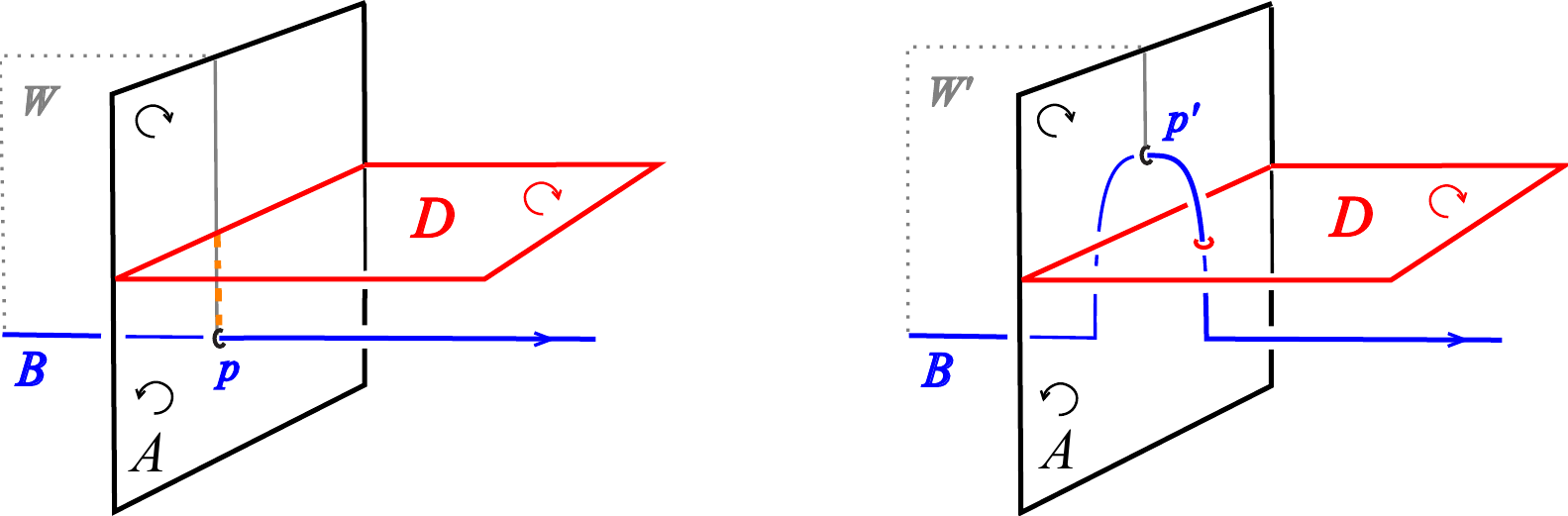}}
         \caption{A `before and after' close-up view of the right-most finger move in Figure~\ref{W-D-ints-and-finger-moves}, but shown in different local coordinates $B^3\times I$: Pushing a self intersection $p$ of $S_t\setminus \partial D$ across $\partial D$ yields a self intersection $p'$  with the opposite sign as $p$, and creates an interior intersection between $S_t$ and $D$ having the same sign as $p$, for $S_t\setminus \partial D$ oriented and $D$ appropriately oriented.}
         \label{push-self-int-across-splitting-disk}
\end{figure}

Finally, each Whitney move as in item~(\ref{item:lem-w-moves-away-from-D}) is guided by Whitney disk whose boundary is disjoint from $\partial D$, hence eliminates a pair of self intersections that have opposite signs using the orientation of $S_t\setminus \partial D$. This completes the proof of Lemma~\ref{nhs} in the simply connected case.

Suppose now that $X\setminus G$ is not simply connected. 
 To establish the  null-homotopy of $D_0\cup D_1$, it will suffice to establish that a certain lift of $D_0\cup D_1$ is null homologous (and therefore null homotopic) in the universal cover of $X\setminus G$. 
By a slight perturbation of $D_0$ with support in the interior, we may assume that $S_1$ is transverse to  the interior of $D_0$. 
By assumption,  a regular neighborhood of $G$ is diffeomorphic to $G\times D^2$. Pick a collection $\{G_k\}$ of distinct parallel spheres;  one sphere $G_k$ for each intersection point in $S_1\cap D_0$. 
 Since $G_k$ and $D_0$ are simply connected, and $S_1$ is $G$-inessential, all lift to copies in the universal cover $\widetilde{X\backslash G}$ of $X\setminus G$. Pick reference lifts $S_1^1$, $G_k^1$ and $D_0^1$ so that $\partial D_0^1\subset S_1^1$ and $G_k^1\cap S_1^1\neq\emptyset$. Let $S_1^\tau$, $G_k^\tau$ and $D_0^\tau$ denote the translates of these chosen lifts by an element $\tau$ of the deck group. The regular homotopy relative to the complement of $U$ lifts to regular homotopies of each $S_1^\tau$ to a surface $S_0^\tau$. As in the simply-connected case define  
\[
\psi(S_t^1,D_0^\tau):=S_t^1\cdot D_0^\tau+\frac12 S_t^1\cdot S_t^\tau.
\] 
One checks that this is invariant under the moves of Lemma~\ref{lem:control-htpy}. Since  $S_1^1$ is embedded, one 
concludes that $S_1^1\cdot D_0^\tau = 0$ and therefore $S_1^\tau\cdot D_0^1 = 0$.  It follows that $D_1^1$ is obtained from $D_0^1$ by adding an algebraically trivial number of copies of  $G_k^\tau$ to $D_0^1$.

\end{proof}

\begin{proof}[Proof of Theorem~\ref{th:6b}]
We first prove the result for surfaces of odd Euler characteristic. By Lemma~\ref{ag1sk} we may isotope $S_1$ to agree with $S_0$ away from a disc neighborhood of a point, and have a further regular homotopy relative to the complement of this neighborhood taking $S_0$ to $S_1$. Without loss of generality, we may assume that the disc neighborhoods  in $S_0$ and $S_1$ are disjoint from $S_0\cap G=S_1\cap G$.
Since the Euler characteristic of $S_0$ is odd, there is a simple closed curve in $S_0$ that has orientable complement and that is disjoint from the disc neighborhood in $S_0$ and from $S_0\cap G$. By the splitting theorem (Theorem~\ref{th:1a}),  this curve bounds a core $D_0$ of a splitting for $S_0$. By Lemma~\ref{nhs}, there is a core  $D_1$ for $S_1$ that agrees with $D_0$ in a neighborhood of $\partial D_1=\partial D_0$ and such that $D_1\cup - D_0$ is null-homotopic.

Let $S_j'$ denote the result of the splitting surgery of $S_j$ along $D_j$ for $j=0,1$. 
It follows that $S_0'$ and $S_1'$ are orientable surfaces in $X$. We claim that $S_0'$ and $S_1'$ are homotopic in $X$. Indeed, let $U_j\subset S_j'$ denote  the complement in $S_j$ to a regular neighborhood of $\partial D_j$. By construction, there is a regular homotopy of $U_0$ to $U_1$ relative to a regular neighborhood of the boundary $\partial U_0=\partial U_1$. On the other hand, since $D_1\cup -D_0$ is null homotopic, the closure of the surface $S_0'\setminus U_0$ is homotopic to the closure of $S_1'\setminus U_1$ relative to the boundary. Therefore $S_0'$ is homotopic to $S_1'$. Furthermore, the surface $S_0$ agrees with the surface $S_1$ in a neighborhood of $S_0\cap G$.  
Thus, $S_0'$ is isotopic to $S_1'$ by the Gabai Theorem (Theorem~\ref{th:5}).   Now $(X,S_0)\cong (X,S_0')\# (S^4,P_\pm)$ and $(X,S_1)\cong (X,S_1')\# (S^4,P_\pm)$ implies that $S_0$ is isotopic to $S_1$ as proved in Lemma~\ref{welld}.

Turn now to the case where $S_j$ is non-orientable with even Euler characteristic. The proof is similar to the odd Euler characteristic case. The difference is that  the complement in $S_0$ of $\partial D_0$ must be non-orientable. Lemma~\ref{nhs1} produces a core of a splitting $D_1$ for $S_1$. Let $S_j'$ denote the result of the splitting of $S_j$ along $D_j$ for $j=0,1$. The surfaces $S_0'$ and $S_1'$ have odd Euler characteristic and  are regularly homotopic. Thus by the odd Euler characteristic case, established in the first part of this proof, we know that $S_0'$ is isotopic to $S_1'$ which using Lemma~\ref{welld} completes the proof.
\end{proof}

\section{Isotopy of surfaces in $4$-manifolds.}\label{s:7}

In view of the Gabai's theorem for non-orientable surfaces (Theorem~\ref{th:6b}), we are in position to extend a recent theorem of \cite{AKMRS} (Theorem~\ref{th:AKMRS}) to the case of non-orientable surfaces. 

\begin{theorem}[Auckly-Kim-Melvin-Ruberman-Schwartz, \cite{AKMRS}] \label{th:AKMRS} If $X$ is a smooth simply-connected $4$-manifold and $\alpha\in  H_2(X)$ is an ordinary class,
then any two closed oriented surfaces $S_0$ and $S_1$ in $X$ of the same genus representing $\alpha$, both with
simply-connected complement, are smoothly isotopic in $X \# (S^2\times S^2)$ (summing away from $S_0 \cup S_1$).
When $\alpha$ is characteristic, the same result holds if one stabilizes by summing with $S^2\tilde\times S^2$. 
\end{theorem}

In Theorem~\ref{th:8a} which extends Theorem~\ref{th:AKMRS} to the case of non-orientable surfaces we assume that $S_0$ is regularly homotopic to $S_1$; note that in a simply connected manifold $X$ orientable surfaces of the same genus representing the same integral homology class are regularly homotopic by Theorem~\ref{l:Gaba}.

\begin{proof}[Proof of Theorem~\ref{th:8a}]
In the case where $S_0$ and $S_1$ are oriented surfaces, Theorem~\ref{th:8a} is established in \cite{AKMRS}. In view of Theorem~\ref{th:6b}, the proof of Theorem~\ref{th:8a} in the non-orientable case follows in the same way as in the orientable case with one minor change.  We provide an abbreviated outline referring to \cite{AKMRS} for longer exposition. As $X\backslash S_0$ is simply-connected, it contains an immersed disk bounded by a meridian of $S_0$. Capping the immersed disc with a fiber of the normal bundle results in an immersed dual $\Sigma$ to $S_0$. If the integral homology class of $S_0$ is ordinary, then, after taking the sum with an immersed sphere disjoint from $S_0$, we may assume that the self-intersection of $\Sigma$ is even.

In contrast to the orientable case, in the present case only the parity of the intersection number $\Sigma\cdot S_0$ is well-defined up to regular homotopy. It follows that $\Sigma\cdot S_1$ is odd. We can reduce the geometric intersection
number $\#(\Sigma\cap S_1)$ down to one. Indeed, if there is excess intersection, pick a pair of intersection points. Given an arc in $S_1$ joining the two intersection points,  one may associate a relative sign to the intersection points via an orientation of  a neighborhood of the path. Since $S_1$ is  non-orientable, one may arrange that the relative sign is negative via a suitable choice of path. 

The proof now continues as in the orientable case in \cite{AKMRS}. Namely, as $X\backslash S_1$ is simply-connected, the union of the chosen path in $S_1$ and a path in $\Sigma$ bounds an immersed disk in the complement of $S_1$. Via finger moves intersections between the disk and $S_0$ may be removed. The correct framing may be obtained by boundary twisting along the portion of the boundary of the disk meeting $\Sigma$. The result is an immersed Whitney disk. Sliding $\Sigma$ across this disk removes a pair of intersection points.

If follows that one may assume that $\Sigma$ intersects $S_0$ and $S_1$, each at exactly one point. In the ordinary case the self-intersection of $\Sigma$ is even. It follows that by taking
 the connected sum of pairs $(X,\Sigma)\# (S^2\times S^2,\{\text{pt}\}\times S^2) = (X\# (S^2\times S^2),\widetilde\Sigma)$ and tubing with copies of $S^2\times \{\text{pt}\}$ one may eliminate the self-intersections of $\widetilde\Sigma$ and adjust the square to zero. The result now follows from the orientable version of the light bulb theorem.
 
 In the characteristic case the self-intersection number of $\Sigma$ is odd. Here one takes the sum  $(X,\Sigma)\# (S^2\widetilde\times S^2,\text{zero section}) = (X\# (S^2\widetilde\times S^2),\widetilde\Sigma)$ to obtain an immersed dual with even square. Tubing with copies of the fiber will remove the self-intersections and adjust the framing to zero. 

\end{proof}


\begin{thebibliography}{99}
\bibitem{AKMRS} D.~Auckly, H. J. Kim, P. Melvin, D. Ruberman, H. Schwartz, Isotopy of surfaces in 4-manifolds after a single stabilization,  Adv. Math. 341 (2019), 609--615.
\bibitem{BT} A.~Bartels, P.~Teichner, All two dimensional links are null homotopic, Geom. Topol. 2 (1999), 235--252.
\bibitem{baykur-sunukjian:surfaces}
R.~I. Baykur and N.~Sunukjian, {\em Knotted surfaces in 4-manifolds and
  stabilizations}, J. Topol., {\bf 9} (2016), 215--231.
\bibitem{Br} E.~H.~Brown Jr, Generalizations of the Kervaire invariant, Ann. Math. 95 (1972), 368--383.
\bibitem{quinn}
F.~Quinn, \emph{Isotopy of 4-manifolds}, J. Diff.\ Geo.
\textbf{24} (1986),
    343--372.
\bibitem{FQ} M.~Freedman, F.~Quinn, Topology of $4$-manifolds, Prinston University Press, Prinston, New Jersey, 1990. 
\bibitem{Ga} D.~Gabai, The $4$-dimensional light bulb theorem,  J. Amer. Math. Soc. 33 (2020), no. 3, 609--652.
\bibitem{GS} R.~Gompf, A.~Stipsicz, 4-manifolds and Kirby calculus. Graduate Studies in Mathematics, 20. American Mathematical Society, Providence, RI, 1999. 558 pp.
\bibitem{Hi59} M.~W.~Hirsch, Immersions of manifolds, Trans. Amer. Math. Soc., 93 (1959), 242-276. 
\bibitem{Ka} S.~Kamada, Nonorientable surfaces in $4$-space, Osaka Journal of Mathematics, 26 (1989), 367-385.
\bibitem{Ko} A.~Kosinski, Differential Manifolds, Academic Press, Boston, 1993.
\bibitem{KM} M.~Kervaire, J.~Milnor. On 2-spheres in 4-manifolds. Proc. Nat. Acad. Sci. U.S.A. 47 (1961), 1651-1657.
\bibitem{Ma69} W.~S.~Massey, Proof of a conjecture of Whitney, Pacific J. Math., 31 (1969), 143-156. 
\bibitem{perron1} 
B. ~Perron, \emph{Pseudo-isotopies et isotopies en dimension 4 dans
la cat\'egorie topologique}, C. R. Acad. Sci. Paris \textbf{299}
(1984), 455--458.
\bibitem{perron2} 
\bysame, \emph{Pseudo-isotopies et isotopies en dimension quatre
dans la
    cat\'egorie topologique}, Topology \textbf{25} (1986), 381--397.   
    \bibitem{PRT} M.~Powell, A.~Ray, P.~Teichner, The $4$-dimensional disc embedding theorem and dual spheres, arXiv:2006.05209. 
\bibitem{Sm} S.~Smale, A classification of immersions of the two-sphere, Trans. AMS (1957), 281-290.
\bibitem{St} N.~Steenrod, The topology of fibre bundles, Princeton University Press, 1951. 
\bibitem{Vi}   O. Ja. Viro, Local knotting of sub-manifolds. (Russian) Mat. Sb. (N.S.) 90 (1973), 173-183.
\end{thebibliography}
\end{document}